\documentclass[11pt]{article}

\usepackage{graphicx}
\usepackage[margin=0.5in]{geometry}
\usepackage{bm}
\usepackage{amsthm}
\usepackage{amscd}
\usepackage{amsbsy}
\usepackage{appendix}
\usepackage{amsmath}
\usepackage{amsfonts}
\usepackage{amssymb}
\usepackage{calligra}
\usepackage{xcolor}
\usepackage{float}
\usepackage[T1]{fontenc}
\usepackage{multirow}
\usepackage{mathrsfs}
\usepackage{mathtools}
\usepackage{epstopdf}
\usepackage{psfrag}
\usepackage{subfigure}
\usepackage{booktabs}
\usepackage{listings}
\usepackage{longtable}
\usepackage{latexsym}
\usepackage{arydshln}
\usepackage{enumerate}
\usepackage{epsfig}
\usepackage{cite}
\usepackage{verbatim}
\usepackage{overpic}
\usepackage{abstract}
\usepackage{extarrows}
\usepackage[pagewise]{lineno}

\allowdisplaybreaks[4]

\date{}
\title{Zero-Hopf Bifurcation of Limit Cycles in \\ Certain Differential Systems}
\author{Bo Huang$^{\text{a},}$\footnote{Corresponding author.} and Dongming Wang$^{\text{b,c}}$\\
	\it\footnotesize $^{\text{a}}$LMIB -- School of Mathematical Sciences,
Beihang University, Beijing 100191, China \\
	\it\footnotesize bohuang0407@buaa.edu.cn\\
	\it\footnotesize $^{\text{b}}$LMIB -- Institute of Artificial Intelligence,
Beihang University, Beijing 100191, China\\
    \it\footnotesize $^{\text{c}}$Centre National de la Recherche Scientifique,
75794 Paris Cedex 16, France\\
	\it\footnotesize Dongming.Wang@lip6.fr}

\newtheorem {theorem*}{Theorem}
\newtheorem{theorem} {Theorem}

\newtheorem{definition}{Definition}

\newtheorem{proposition}{Proposition}
\newtheorem{lemma}{Lemma}
\newtheorem{corollary}{Corollary}
\newtheorem{remark}{Remark}

\newtheorem{conjecture} {Conjecture}

\newtheorem{open problem} {Open problem}

\numberwithin{equation}{section}

\usepackage[colorlinks,
            CJKbookmarks=true,
            linkcolor=blue,
            anchorcolor=red,
             urlcolor=blue,
            citecolor=blue
            ]{hyperref}

\begin{document}
\maketitle
\noindent {\bf Abstract.}  This paper studies the number of limit cycles that may bifurcate from an equilibrium of an autonomous system of differential equations. The system in question is assumed to be of dimension $n$, have a zero-Hopf equilibrium at the origin, and consist only of homogeneous terms of order $m$. Denote by $H_k(n,m)$ the maximum number of limit cycles of the system that can be detected by using the averaging method of order $k$. We prove that $H_1(n,m)\leq(m-1)\cdot m^{n-2}$ and $H_k(n,m)\leq(km)^{n-1}$ for generic $n\geq3$, $m\geq2$ and $k>1$. The exact numbers of $H_k(n,m)$ or tight bounds on the numbers are determined by computing the mixed volumes of some polynomial systems obtained from the averaged functions. Based on symbolic and algebraic computation, a general and algorithmic approach is proposed to derive sufficient conditions for a given differential system to have a prescribed number of limit cycles. The effectiveness of the proposed approach is illustrated by a family of third-order differential equations and by a four-dimensional hyperchaotic differential system.


\smallskip

\noindent {\bf Math Subject Classification (2020).} 34C07; 37G15; 68W30.

\smallskip


\section{Introduction} \label{sect1}

In a nonlinear system of ordinary differential equations there may be exhibited a certain kind of qualitative behavior called limit cycles near a singularity of the system. A limit cycle of a differential system is an isolated periodic orbit of the system. The concept of limit cycles was introduced by Poincar\'e \cite{poi1881} in the late 1800s when he observed the behavior of limit cycles and built up the qualitative theory of differential equations. Since then the analysis of limit-cycle bifurcation has been long-standing as a challenging problem under extensive investigation. In particular, determination of the number and relative positions of limit cycles for planar polynomial differential systems is known as the second part of Hilbert's 16th problem \cite{DH1902} that is still open.

In connection with Hilbert's problem the literature is vast (see \cite{Ily2002,li2003} and references therein). The qualitative theory of differential equations established by Poincar\'e \cite{poi1881} and Lyapunov \cite{lya1892} was developed further by Bendixson \cite{Ben1901}, Andronov \cite{Andr1973}, Arnold \cite{Arn1983}, and many others for qualitative analysis of global stability, local stability of equilibria, and bifurcation of limit cycles for various classes of differential systems. In the long history of development there were published a few results on the number of limit cycles. For instance, in a paper of 1955, Petrovsky and Landis \cite{PeLa55} attempted to prove that planar quadratic differential systems can have at most 3 limit cycles. Surprisingly, in 1979, Chen and Wang \cite{ChWa79} and Shi \cite{shi80} independently  found counter-examples to the result of Petrovskii and Landis, showing that planar quadratic differential systems can have 4 limit cycles. These unexpected, inconsistent results have stimulated a lot of interest in the research of limit cycles for quadratic and cubic systems. In particular, Li and others \cite{LLY09} proved that planar cubic differential systems may have at least 13 limit cycles, the best lower bound known so far. From the work of these and many other authors the reader can find a large variety of inspiring results on the number of limit cycles for various classes of differential systems. These results have not only helped enrich the classical qualitative theory of differential equations considerably, but also found important applications in modern science and engineering.

We consider $\mathcal{C}^{m+1}$ differential systems in $\mathbb{R}^n$ with $n\geq3$ and $m\geq2$, and assume that the systems have an isolated singularity at the origin, so that their Taylor expansion can be written in the form
\begin{equation}\label{eq1}
\dot{x}_s=\sum_{j=1}^mP_{s,j}(x_1,\ldots,x_n)+Q_s(x_1,\ldots,x_n),\quad
s=1,\ldots,n,
\end{equation}
where each $P_{s,j}$ is a homogeneous polynomial of order $j$ in $x_1,\ldots,x_n$ and $Q_{s}$ is a Taylor series starting with terms of order $>m$. Let $M$ be the coefficient matrix of the linear terms in $P_{s,1}$. When $M$ has a pair of purely imaginary eigenvalues (e.g., $\pm bi$) and the other eigenvalues are non-zero, we call the origin a Hopf equilibrium point. If some of the eigenvalues of $M$ other than $\pm bi$ are zero, then the origin is called a zero-Hopf equilibrium point. Limit cycles may bifurcate from Hopf equilibria in nonlinear differential systems of the form \eqref{eq1} as the values of the parameters in the coefficients vary. A generic Hopf bifurcation is a bifurcation of a limit cycle from a Hopf equilibrium and a zero-Hopf bifurcation is a bifurcation of a limit cycle from a zero-Hopf equilibrium. Here we are interested in zero-Hopf bifurcations where all the eigenvalues of $M$ different from $\pm bi$ are zero; such kind of zero-Hopf bifurcation is called \textit{complete} zero-Hopf bifurcation. Zero-Hopf bifurcations for three-dimensional systems have been extensively studied (see \cite{LBS07,LMB09,LV11,JA20,YK2004,GH1993} and references therein). Our investigations in this paper are focused on \textit{complete} zero-Hopf bifurcations for differential systems of the form \eqref{eq1} with
$P_{s,2}=\cdots=P_{s,m-1}=0$ for $s=1,\ldots,n$; in this case, system \eqref{eq1} can be written in the form
\begin{equation}\label{equ-1}
\begin{split}
& \dot{x}_1=-bx_2+P_{1,m}(x_1,\ldots,x_n)+Q_1(x_1,\ldots,x_n),\\
& \dot{x}_2=bx_1+P_{2,m}(x_1,\ldots,x_n)+Q_2(x_1,\ldots,x_n),\\
& \dot{x}_s=P_{s,m}(x_1,\ldots,x_n)+Q_s(x_1,\ldots,x_n),\quad s=3,\ldots,n,
\end{split}
\end{equation}
where
\begin{equation}\label{equ-1-0}
P_{s,m}=\sum_{i_1+\cdots+i_n=m}p_{s,i_1,i_2,\ldots,i_n}x_1^{i_1}x_2^{i_2}\cdots x_n^{i_n},\quad s=1,\ldots,n,\nonumber
\end{equation}
$p_{s,i_1,i_2,\ldots,i_n}$ are real parameters, and $b\neq0$.

The main goal of this paper is to determine how many limit cycles can bifurcate from the origin, as a zero-Hopf equilibrium of system \eqref{equ-1}, when the system is perturbed inside the class of differential systems of the same form. We shall use the method of averaging \cite{LNT14}, one of the most powerful and widely used methods for limit-cycle bifurcation in the presence of a small parameter $\varepsilon$. According to the theory of averaging one can replace a vector field by its average (over time or an angular variable) to obtain asymptotic approximations to the original system such that the existence of limit cycles is guaranteed in the approximations. The averaging method permits us to reduce the problem of studying limit cycles of differential systems to that of finding common isolated zeros of the resulting averaged functions. We refer to the books \cite{SVM07,LMS15} for a modern exposition of the averaging theory. It is known that center manifold theory and normal form theory are also powerful tools for the analysis of zero-Hopf bifurcations of nonlinear differential systems (see \cite{GH1993,YK2004,HY2012}). Algorithms have been developed for the computation of center manifolds and normal forms \cite{BiYu99,TiYu14}, but they give no qualitative information about the bifurcated limit cycles. In contrast, using averaged functions, one can determine the shape of the bifurcated limit cycles up to any order in the parameter $\varepsilon$, see \cite{JLXZ09,LMB09}. On the other hand, it has been shown in \cite{BPY20} that the averaging method may be unable to detect possible limit cycles bifurcating from a zero-Hopf equilibrium, while the normal form theory can be used to overcome the difficulty.

To apply the averaging method, one usually considers certain perturbations of the involved parameters in \eqref{equ-1} using a small perturbation variable $\varepsilon$, see \cite{LMB09,DBM21} for instance. Previous results concerning zero-Hopf bifurcations of $\mathcal{C}^{m+1}$ differential systems were obtained mainly for the first order in $\varepsilon$ (see \cite{JLXZ09,PZ2009,BLV18}) and at most up to the third order (see \cite{LBS07,LMB09,LV11,BBM17,DBM21} for studies on polynomial differential systems of the form \eqref{equ-1} of low dimension and low degree: quadratic or cubic). There are also a few results about zero-Hopf bifurcations for high-dimensional differential systems of arbitrary degree. One of them is due to Pi and Zhang \cite{PZ2009} who studied zero-Hopf bifurcations of $\mathcal{C}^{m+1}$ differential systems for general $m$ in $\mathbb{R}^n$ up to the first order in $\varepsilon$. By means of higher-order averaging, some of the results on periodic solutions may be improved qualitatively and quantitatively (see \cite{GGL2013,DBM21}).

The class of perturbed systems studied in this paper is obtained by adding to system \eqref{equ-1} perturbation terms of order up to $k$ in $\varepsilon$, taking the following form
\begin{equation}\label{equ-3}
\begin{split}
\dot{x}_1&=\sum_{j=1}^ka_j\varepsilon^jx_1-\big(b+\sum_{j=1}^kb_j\varepsilon^j\big)x_2
+\sum_{i_1+\cdots+i_n=m}p_{1,i_1,i_2,\ldots,i_n}x_1^{i_1}x_2^{i_2}\cdots x_n^{i_n}+Q_1,\\
\dot{x}_2&=\big(b+\sum_{j=1}^kb_j\varepsilon^j\big)x_1+\sum_{j=1}^ka_j\varepsilon^jx_2
+\sum_{i_1+\cdots+i_n=m}p_{2,i_1,i_2,\ldots,i_n}x_1^{i_1}x_2^{i_2}\cdots x_n^{i_n}+Q_2,\\
\dot{x}_s&=\sum_{j=1}^kc_{s,j}\varepsilon^jx_s+\sum_{i_1+\cdots+i_n=m}p_{s,i_1,i_2,\ldots,i_n}x_1^{i_1}x_2^{i_2}\cdots x_n^{i_n}+Q_s,\quad s=3,\ldots,n,
\end{split}
\end{equation}
where
\begin{equation}\label{equ-3-0}
p_{s,i_1,i_2,\ldots,i_n}=\sum_{j=1}^k\varepsilon^{j-1} p_{s,i_1,i_2,\ldots,i_n,j-1},\quad s=1\ldots,n,\nonumber
\end{equation}
$a_j$, $b_j$, $c_{s,j}$, $p_{s,i_1,i_2,\ldots,i_n,j-1}$
are real parameters, and $\varepsilon$ is the perturbation parameter. The coefficient matrix of the linear part of system \eqref{equ-3} has eigenvalues $(\sum_{j=1}^ka_j\varepsilon^j)\pm i(b+\sum_{j=1}^kb_j\varepsilon^j)$ and $\sum_{j=1}^kc_{s,j}\varepsilon^j$ for $s=3,\ldots,n$. Note that the class of systems can be enlarged by adding more linear perturbation terms to the system \eqref{equ-1} (see Section \ref{sec-hyper}). The interest in the class of systems comes for two main reasons: one is its simplicity as all the eigenvalues of the coefficient matrix of its linear terms can be written out clearly, and the other is that system \eqref{equ-3} has been used to model a large variety of real-world phenomena and as a test bed for studies on the bifurcation problem (see \cite{JLXZ09,BLV18,LBS07,DBM21} for instance).

Let $H_k(n,m)$ denote the maximum number of limit cycles of the differential system in question that can be detected by using the $k$th-order averaging method. This number equals the exact number of real isolated zeros of the averaged functions. It is not necessarily the maximum number of limit cycles that the systems may have. In this paper, we show how to determine the exact number of limit cycles and provide bounds on the number $H_k(n,m)$ of limit cycles which can bifurcate from the origin of system \eqref{equ-3} with $Q_1=\cdots=Q_s=0$ for $|\varepsilon|>0$ sufficiently small. Most of the results obtained can be easily extended to $\mathcal{C}^{m+1}$ differential systems in $\mathbb{R}^n$, but for the sake of simplicity we choose to present the results only for polynomial differential systems. More concretely, the following problem will be addressed.
\begin{description}
\item[Problem {\rm (on the numbers and the conditions for the existence of limit cycles)}]
  \item[A.] Determine the number $H_k(n,m)$ for given integers $n$ and $m$.
  \item[B.] Provide lower and upper bounds of $H_k(n,m)$ for generic $n$ and $m$.
  \item[C.] Derive conditions under which system \eqref{equ-3} has a prescribed number of limit cycles.
\end{description}

The method of averaging has been studied extensively both in theory and for applications. Here we recall some of the known results that are particularly relevant to the problem. Llibre and Zhang proved in \cite{JLXZ09} that up to the first order in $\varepsilon$, $H_1(n,2)=2^{n-3}$ limit cycles can bifurcate from the origin of a perturbed system for which the coefficient matrix of the linear part has eigenvalues $\varepsilon a\pm bi$ and $\varepsilon c_s$ for $s=3,\ldots,n$. They showed for the first time that the number of bifurcated limit cycles can grow exponentially with the dimension $n$, and applied their results to certain fourth-order differential equations as well as a simplified Marchuk model that describes immune response. Later, Pi and Zhang \cite{PZ2009} extended the result of \cite{JLXZ09} for $m=2$ to that for general $m\geq2$, showing that either $H_1(n,m)\geq m^{n-2}/2$ for $m$ even or $H_1(n,m)\geq m^{n-2}$ for $m$ odd. This is the first result showing that the number of limit cycles from a zero-Hopf bifurcation is a power function in the degree of the system. Recently, Barreira and others \cite{BLV18} proved that $H_1(n,3)=3^{n-2}$, $H_1(n,4)\leq6^{n-2}$, and $H_1(n,5)\leq4\cdot5^{n-2}$. In particular, $H_1(3,4)=2$ and $H_1(3,5)=5$. Some detailed analysis for special situations of system \eqref{equ-3} in low dimensions may be found in \cite{LBS07,LMB09,LV11,BBM17}.

The investigations in the present paper are based mainly on the averaging method and some algebraic methods with exact symbolic computation. Our solution to Problem A is a systematical method with interactive algebraic computation, to Problem B is a method based on B\'ezout's bound and Bernstein's bound for counting the number of common isolated zeros of a polynomial system, and to Problem C is an effective algorithmic approach with polynomial-algebra methods to solve semi-algebraic systems. Using our solutions to the problems, we establish two main results on the bounds of the number $H_k(n,m)$ for generic $n$ and $m$: one upper bound $H_1(n,m)\leq(m-1)\cdot m^{n-2}$ up to the first order in $\varepsilon$ (which can be made tighter by Bernstein's theorem, see Theorem \ref{main-th-2}) and the other upper bound $H_k(n,m)\leq(km)^{n-1}$ up to the $k$th order in $\varepsilon$. The former is relatively smaller than the latter when $k=1$, because we can derive explicit expressions of the first-order averaged functions for generic $n$ and $m$ (see Section \ref{sect-main-1}). The forms of higher-order averaged functions can be determined by induction over the averaging order $k$. The induction process is detailed in Section \ref{sec-B}.

The rest of this paper is organized as follows. We explain the basic theory of the averaging method and introduce the concept of mixed volume and Bernstein's theorem in Section \ref{sect2-pre}. The main results (solutions to Problems A, B and C), Theorem \ref{main-th-0} and Corollary \ref{main-co-1} for the bound on the number of limit cycles of system \eqref{equ-3} up to the $k$th-order averaging, Theorems \ref{main-th-1} and \ref{main-th-2} for the bounds on the number of limit cycles of system \eqref{equ-3} up to the first-order averaging, and Theorem \ref{semi-averaging} for a given parametric differential system of the form \eqref{equ-3} to have prescribed numbers of limit cycles are presented in Section \ref{sect3-0}. In Sections \ref{sect-diff} and \ref{sec-hyper}, we demonstrate the effectiveness of our computational approach using results obtained for a class of third-order differential equations and a four-dimensional hyperchaotic differential system. The proofs of Theorems \ref{main-th-0} and \ref{main-th-1} are given in Sections \ref{sec-B} and \ref{sect-main-1} respectively.

\section{Averaging Method for Bifurcation of Limit Cycles and Zero Bounds of Polynomial Systems}\label{sect2-pre}
This section presents some results that will be used to study limit cycles for high-dimensional differential systems. Section \ref{sect2-1} is dedicated to the averaging method for studying periodic solutions of a differential system in $\mathbb{R}^n$ at any order in the small parameter $\varepsilon$. In Section \ref{sec-BBK} we recall, from convex algebraic geometry, the concept of mixed volume and Bernstein's theorem, which we apply to predict the number of solutions of a generic polynomial system.

\subsection{Averaging Method and Its Underlying Theory}\label{sect2-1}
The averaging method introduced here is a slight modification of \cite{LNT14} (and \cite{DDN17} at one point). It deals with differential systems in the following standard form
\begin{equation}\label{equ2-1}
\begin{split}
\frac{d\boldsymbol{x}}{dt}=\sum_{i=0}^k\varepsilon^i\boldsymbol{F}_i(t,\boldsymbol{x})+\varepsilon^{k+1}\boldsymbol{R}(t,\boldsymbol{x},\varepsilon),
\end{split}
\end{equation}
where $\boldsymbol{F}_i:\mathbb{R}\times D\rightarrow\mathbb{R}^n$ for $i=0,1,\ldots,k$ and $\boldsymbol{R}:\mathbb{R}\times D\times(-\varepsilon_0,\varepsilon_0)\rightarrow\mathbb{R}^n$ are continuous functions and $T$-periodic in the variable $t$, with $D$ being an open subset of $\mathbb{R}^n$, and $0<\varepsilon_0\ll1$ a sufficiently small parameter. Let $L$ be a positive integer, $\boldsymbol{x}=(x_1,\ldots,x_n)\in D$, $t\in\mathbb{R}$, and $\boldsymbol{y}_j=(y_{j1},\ldots,y_{jn})\in\mathbb{R}^n$ for $j=1,\ldots,L$. For any sufficiently smooth function $\boldsymbol{F}:\mathbb{R}\times D\rightarrow\mathbb{R}^n$ and $(t,\boldsymbol{x})\in\mathbb{R}\times D$, we denote by $\partial^L\boldsymbol{F}(t,\boldsymbol{x})$ a symmetric $L$-multilinear map which is applied to a product of $L$ vectors of $\mathbb{R}^n$, denoted as $\bigodot_{j=1}^L\boldsymbol{y}_j\in\mathbb{R}^{nL}$. This $L$-multilinear map is defined formally as
\begin{equation}\label{equ2-2}
\begin{split}
\partial^L\boldsymbol{F}(t,\boldsymbol{x})\bigodot_{j=1}^L\boldsymbol{y}_j=\sum_{1\leq i_1,\ldots,i_L\leq n}
\frac{\partial^L\boldsymbol{F}(t,\boldsymbol{x})}{\partial x_{i_1}\cdots\partial x_{i_L}}y_{1i_1}\cdots y_{Li_L}.
\end{split}
\end{equation}

\begin{remark}\label{remk2-1}
The $L$-multilinear map defined in \eqref{equ2-2} is the $L$th Fr\'echet derivative of the function $\boldsymbol{F}(t,\boldsymbol{x})$ with respect to the variable $\boldsymbol{x}$. For any positive integer $b$ and vector $\boldsymbol{y}\in\mathbb{R}^n$, we write $\boldsymbol{y}^b=\bigodot_{i=1}^b\boldsymbol{y}\in\mathbb{R}^{nb}$.
\end{remark}

The averaging method works by defining a collection of functions $\boldsymbol{f}_i:D\rightarrow\mathbb{R}^n$, each $\boldsymbol{f}_i$ called the $i$th-order averaged function, for $i=1,2,\ldots,k$, whose simple zeros control, for $|\varepsilon|>0$ sufficiently small, the limit cycles of system \eqref{equ2-1}. It has been established by Llibre and others in \cite{LNT14} that
\begin{equation}\label{equ2-3}
\begin{split}
\boldsymbol{f}_i(\boldsymbol{z})=\frac{\boldsymbol{y}_i(T,\boldsymbol{z})}{i!},
\end{split}
\end{equation}
where $\boldsymbol{y}_i:\mathbb{R}\times D\rightarrow\mathbb{R}^n$, for $i=1,2,\ldots,k$, are defined recursively by the following integral equations
\begin{equation}\label{equ2-4}
\begin{split}
\boldsymbol{y}_1(t,\boldsymbol{z})&=\int_0^t\boldsymbol{F}_1(\theta,\boldsymbol{z})d\theta,\\
\boldsymbol{y}_i(t,\boldsymbol{z})&=i!\int_0^t\Big(\boldsymbol{F}_i(\theta,\boldsymbol{z})
+\sum_{\ell=1}^{i-1}\sum_{S_{\ell}}\frac{1}{b_1!b_2!2!^{b_2}\cdots b_{\ell}!\ell!^{b_{\ell}}}\partial^L\boldsymbol{F}_{i-\ell}(\theta,\boldsymbol{z})
\bigodot_{j=1}^{\ell}\boldsymbol{y}_j(\theta,\boldsymbol{z})^{b_j}\Big)d\theta.
\end{split}
\end{equation}
For the sake of simplicity, in \eqref{equ2-4}, we assume that $\boldsymbol{F}_0=0$. The symbol $S_{\ell}$ is the set of all $\ell$-tuples $[b_1,b_2,\ldots,b_{\ell}]$ of nonnegative integers satisfying $b_1+2b_2+\cdots+\ell b_{\ell}=\ell$ and $L=b_1+b_2+\cdots+b_{\ell}$. Note that, related to the averaged function \eqref{equ2-3} there exist two fundamentally different cases of system \eqref{equ2-1}, namely, when $\boldsymbol{F}_0=0$ and when $\boldsymbol{F}_0\neq0$. We see that when $\boldsymbol{F}_0\neq0$, the formula for $\boldsymbol{y}_i(t,\boldsymbol{z})$ in \eqref{equ2-4} requires the solution of a Cauchy problem for $i=1,2,\ldots,n$ (see Remark 3 in \cite{LNT14}). The investigation in this paper is restricted to the case where $\boldsymbol{F}_0=0$.

The following averaging theorem provides a criterion for the existence of limit cycles. Its proof can be found in \cite{LNT14}.
\begin{theorem}\label{averaging-thm}
Assume that the following conditions hold for any $k\geq1$:
\begin{enumerate}
  \item[\emph{(a)}] for each $i=1,2,\ldots,k$ and $t\in\mathbb{R}$, the function $\boldsymbol{F}_i(t,\boldsymbol{x})$ is of class $\mathcal{C}^{k-i}$, $\partial^{k-i}\boldsymbol{F}_i$ is locally Lipschitz in $\boldsymbol{x}$, and $\boldsymbol{R}$ is a continuous function locally Lipschitz in $\boldsymbol{x}$;
  \item[\emph{(b)}] for some $j\in\{1,2,\ldots,k\}$, $\boldsymbol{f}_i=0$ for $i=1,2,\ldots,j-1$ and $\boldsymbol{f}_j\neq0$;
  \item[\emph{(c)}] for some $\boldsymbol{z}^*\in D$ with $\boldsymbol{f}_j(\boldsymbol{z}^*)=0$ we have ${\rm{det}}(J_{\boldsymbol{f}_j}(\boldsymbol{z}^*))\neq0$.

\end{enumerate}
Then, for any $|\varepsilon|>0$ sufficiently small, there exists a $T$-periodic solution $\boldsymbol{x}(t,\varepsilon)$ of \eqref{equ2-1} such that $\boldsymbol{x}(0,\varepsilon)\rightarrow\boldsymbol{z}^*$ when $\varepsilon\rightarrow0$.
\end{theorem}

\begin{remark}\label{remk2-2}
The notation ${\rm{det}}(J_{\boldsymbol{f}_j}(\boldsymbol{z}))\neq0$ means that the Jacobian determinant of $\boldsymbol{f}_j$ at $\boldsymbol{z}\in V$ is non-zero. In Theorem \ref{averaging-thm}, the function $\boldsymbol{f}_j$ for $j\in\{1,2,\ldots,k\}$ (defined in \eqref{equ2-3}) is assumed to be a $\mathcal{C}^1$ function. In this case, instead of Brouwer degree theory, the implicit function theorem could be used to
prove Theorem \ref{averaging-thm}, see \cite[Remark 4]{LNT14}.

\end{remark}

The assumption (a) of Theorem \ref{averaging-thm} assures the existence and uniqueness of the solution of \eqref{equ2-1} for each initial value on the interval $[0,T]$. Hence, for each $\boldsymbol{z}\in D$ we can let $\boldsymbol{x}(t,\boldsymbol{z},\varepsilon)$ be the solution of \eqref{equ2-1} such that $\boldsymbol{x}(0,\boldsymbol{z},\varepsilon)=\boldsymbol{z}$. For $\varepsilon_0>0$ sufficiently small, we consider the function $\boldsymbol{\zeta}:D\times(-\varepsilon_0,\varepsilon_0)\rightarrow\mathbb{R}^n$ defined by
\begin{equation}\label{equ2-7}
\begin{split}
\boldsymbol{\zeta}(\boldsymbol{z},\varepsilon)=\int_0^T\Big[\sum_{i=1}^k\varepsilon^i
\boldsymbol{F}_i(t,\boldsymbol{x}(t,\boldsymbol{z},\varepsilon))
+\varepsilon^{k+1}\boldsymbol{R}(t,\boldsymbol{x}(t,\boldsymbol{z},\varepsilon),\varepsilon)\Big]dt.\nonumber
\end{split}
\end{equation}
It follows from \eqref{equ2-1} that, for every $\boldsymbol{z}\in D$,
\begin{equation}\label{equ2-8}
\begin{split}
\boldsymbol{\zeta}(\boldsymbol{z},\varepsilon)=\boldsymbol{x}(T,\boldsymbol{z},\varepsilon)-\boldsymbol{x}(0,\boldsymbol{z},\varepsilon).\nonumber
\end{split}
\end{equation}
Therefore, finding a $T$-periodic solution $\boldsymbol{x}(t,\boldsymbol{z}_{\varepsilon},\varepsilon)$ of \eqref{equ2-1} is equivalent to finding a solution $\boldsymbol{z}_{\varepsilon}$ of $\boldsymbol{\zeta}(\boldsymbol{z},\varepsilon)=0$. Note that the function $\boldsymbol{\zeta}(\boldsymbol{z},\varepsilon)$ can be written in the form
\begin{equation}\label{equ2-9}
\begin{split}
\boldsymbol{\zeta}(\boldsymbol{z},\varepsilon)=\varepsilon \boldsymbol{f}_1(\boldsymbol{z})+\varepsilon^2 \boldsymbol{f}_2(\boldsymbol{z})+\cdots+\varepsilon^k \boldsymbol{f}_k(\boldsymbol{z})+\mathcal{O}(\varepsilon^{k+1}),\nonumber
\end{split}
\end{equation}
where the function $\boldsymbol{f}_i(\boldsymbol{z})$ is the one defined in \eqref{equ2-3} for $i=1,2,\ldots,k$.


In practical terms, the evaluation of the recurrence \eqref{equ2-4} is a computational problem. Recently, Novaes \cite{DDN17} used the partial Bell polynomials to provide an alternative formula for the recurrence. This formula can make the computational implementation of the averaged functions \eqref{equ2-3} easier. A partial Bell polynomial is expressed by
\begin{equation}\label{equ2-5}
\begin{split}
B_{\ell,m}(x_1,\ldots,x_{\ell-m+1})=\sum_{\widetilde{S}_{\ell,m}}\frac{\ell!}{b_1!b_2!\cdots b_{\ell-m+1}!}\prod_{j=1}^{\ell-m+1}\left(\frac{x_j}{j!}\right)^{b_j},
\end{split}
\end{equation}
where $\ell$ and $m$ are positive integers, $\widetilde{S}_{\ell,m}$ is the set of all $(\ell-m+1)$-tuples $[b_1,b_2,\ldots,b_{\ell-m+1}]$ of nonnegative integers satisfying $b_1+2b_2+\cdots+(\ell-m+1)b_{\ell-m+1}=\ell$, and $b_1+b_2+\cdots+b_{\ell-m+1}=m$.

The following result provides an alternative formula for the averaged functions.

\begin{theorem}\label{thm2-a}
For $i=1,2,\ldots,k$ the recurrence \eqref{equ2-4} reads
\begin{equation}\label{equ2-6}
\begin{split}
\boldsymbol{y}_1(t,\boldsymbol{z})&=\int_0^t\boldsymbol{F}_1(\theta,\boldsymbol{z})d\theta,\\
\boldsymbol{y}_i(t,\boldsymbol{z})&=i!\int_0^t\Big(\boldsymbol{F}_i(\theta,\boldsymbol{z})
+\sum_{\ell=1}^{i-1}\sum_{m=1}^{\ell}\frac{1}{\ell!}\partial^m\boldsymbol{F}_{i-\ell}(\theta,\boldsymbol{z})
B_{\ell,m}\big(\boldsymbol{y}_1(\theta,\boldsymbol{z}),\ldots,\boldsymbol{y}_{\ell-m+1}(\theta,\boldsymbol{z})\big)\Big)d\theta.
\end{split}
\end{equation}

\end{theorem}

The process of using the averaging method for studying limit cycles of differential systems can be divided into three steps \cite[Section 4]{HY21}. A slightly modified version of them is listed as follows.

\begin{description}
  \item[Step 1.] Write a perturbed system of the form \eqref{equ-3} in the standard form of averaging \eqref{equ2-1} up to the $k$th order in $\varepsilon$.
  \item[Step 2.] (a) Compute the exact formula of the $k$th-order integral function $\boldsymbol{y}_k(t,\boldsymbol{z})$ in \eqref{equ2-6}; (b) Derive the symbolic expression of the $k$th-order averaged function $\boldsymbol{f}_k(\boldsymbol{z})$ by \eqref{equ2-3}.
  \item[Step 3.] Determine the exact upper bound for the number of real isolated zeros of $\boldsymbol{f}_k(\boldsymbol{z})$.

\end{description}

The averaging method allows one to find periodic solutions of periodic non-autonomous differential systems (see \eqref{equ2-1}). However we are interested in using it for analyzing limit cycles bifurcating from a zero-Hopf equilibrium of the autonomous differential system \eqref{equ-3}. The determination of the exact number of isolated zeros of the averaged functions \eqref{equ2-3} up to every order is usually very difficult to be done. But for our zero-Hopf bifurcation analysis, the main work is to study the maximum number of real solutions of a polynomial system obtained from the averaged functions. Some advanced techniques from symbolic computation, such as Gr\"obner basis \cite{BB85}, triangular decomposition \cite{WTW00,DW01}, quantifier elimination \cite{GEC75,CHH91}, and real solution classification \cite{YHX01,LYBX05} may be used to perform the task. More detailed discussions of the averaging method, including applications, can be found in \cite{ABJL2004,SVM07,GGL2013,LMS15}.

\subsection{Mixed Volumes and Bernstein's Theorem}\label{sec-BBK}
We recall, from convex geometry, the concept of mixed volume and Bernstein's theorem, which we will apply to estimate the number of zeros of averaged functions. For background on polytopes and Minkowski sums, we refer the reader to the books \cite{GMZ1995} and \cite[Section 7]{DJD05}.

Let $\boldsymbol{e}=(e_1,\ldots,e_n)\in\mathbb{Z}^n$ be an integer vector. The corresponding \textit{Laurent monomial} in variables $\boldsymbol{x}=(x_1,\ldots,x_n)$ is $\boldsymbol{x}^{\boldsymbol{e}}=x_1^{e_1}\cdots x_n^{e_n}$. A Laurent polynomial $p$ is given by
\[p(\boldsymbol{x})=\sum_{\boldsymbol{e}\in\mathcal{A}}c_{\boldsymbol{e}}\boldsymbol{x}^{\boldsymbol{e}},\]
where $c_{\boldsymbol{e}}\in\mathbb{C}$ and $\mathcal{A}$, the support of $p$, is a finite subset of $\mathbb{Z}^n$. In short, $p$ is an element of the ring $\mathbb{C}[x_1^{\pm},\ldots,x_n^{\pm}]$. A Laurent polynomial system $P(\boldsymbol{x})=(p_1(\boldsymbol{x}),\ldots,p_n(\boldsymbol{x}))$ is an $n$-tuple of nonzero Laurent polynomials. Let $\mathcal{A}_i$ be the support of $p_i(\boldsymbol{x})$, for $i=1,\ldots,n$.

A polytope is a nonempty compact convex set with finitely many extremal points, called vertices.

\begin{definition}\label{NP-1}
The Newton polytope of $p_i$, denoted $Q_i={\rm{conv}}(\mathcal{A}_i)\subset\mathbb{R}^n$, is the convex hull of support $\mathcal{A}_i$.
\end{definition}

There are two operations induced by the vector space structure in $\mathbb{R}^n$ that form new polytopes from old ones.

\begin{definition}\label{MK-sum}
Let $P,Q$ be polytopes in $\mathbb{R}^n$ and let $\lambda\geq0$ be a real number.
\begin{itemize}
  \item[\emph{(a)}] The \textit{Minkowski sum} of $P$ and $Q$, denoted $P+Q$, is
      \[P+Q=\{p+q: p\in P~~\text{and}~~q\in Q\},\]
      where $p+q$ denotes the usual vector sum in $\mathbb{R}^n$.
  \item[\emph{(b)}] The polytope $\lambda P$ is defined by $\lambda P=\{\lambda p: p\in P\}$, where $\lambda p$ is the usual scalar multiplication on $\mathbb{R}^n$.
\end{itemize}

\end{definition}

Let ${\rm{Vol}}(A)$ denote the usual Euclidean volume of polytope $A\subset\mathbb{R}^n$. The next result concerns the volumes of linear combinations of polytopes formed according to Definition \ref{MK-sum}.

\begin{proposition}\label{prop-poly-1}
Consider any collection of polytopes $P_1,\ldots,P_r$ in $\mathbb{R}^n$, and let $\lambda_1,\ldots,\lambda_r\in\mathbb{R}$ be nonnegative. Then ${\rm{Vol}}(\lambda_1P_1+\cdots+\lambda_rP_r)$ is a homogeneous polynomial function of degree $n$ in the $\lambda_i$.
\end{proposition}

When $r=n$, we can single out one particular term in the polynomial ${\rm{Vol}}(\lambda_1P_1+\cdots+\lambda_nP_n)$ that has special meaning for the whole collection of polytopes.

\begin{definition}\label{MV-1}
The $n$-dimensional \textit{mixed volume} of a collection of polytopes $P_1,\ldots,P_n$, denoted ${\rm{MV}}(P_1,\ldots,P_n)$, is the coefficient of the monomial $\lambda_1\lambda_2\cdots\lambda_n$ in ${\rm{Vol}}(\lambda_1P_1+\cdots+\lambda_nP_n)$.
\end{definition}

An explicit expression for the mixed volume is obtained by the Exclusion-Inclusion principle:
\[{\rm{MV}}(P_1,\ldots,P_n)=\sum_{I\subset\{1,\ldots,n\}}(-1)^{n-|I|}{\rm{Vol}}(\sum_{i\in I}P_i),\]
where $|\cdot|$ denotes set cardinality and $\sum_{i\in I}P_i$ is the Minkowski sum of polytopes.

Let $\mathbb{C}^*=\mathbb{C}\backslash\{\boldsymbol{0}\}$ denote the set of nonzero complex numbers. We are now ready to state Bernstein's theorem on the number of roots in $(\mathbb{C}^*)^n$ for a system of $n$ polynomials.

\begin{theorem}[\cite{Ber75}]\label{Ber-thm}
For any given Laurent polynomials $p_1,\ldots,p_n$ over $\mathbb{C}$ with Newton polytopes $Q_1,\ldots,Q_n\subset\mathbb{R}^n$, the number of common isolated zeros in $(\mathbb{C}^*)^n$, counted with multiplicities, is bounded by the mixed volume ${\rm{MV}}(Q_1,\ldots,Q_n)$. Moreover, for generic choices of the coefficients in the $p_i$, the number of isolated zeros is exactly ${\rm{MV}}(Q_1,\ldots,Q_n)$.
\end{theorem}

The mixed volume ${\rm{MV}}(Q_1,\ldots,Q_n)$ is sometimes called the \textit{BKK bound}, as it relies heavily on the work by Kushnirenko \cite{Kus76} and has been alternatively proven by Khovanskii \cite{Kho77}. The following result is due to Li and Wang \cite{LW96}. It extends Bernstein's theorem to provide an upper bound on the number of common zeros in $\mathbb{C}^n$.

\begin{theorem}\label{LW-thm}
Let $p_1,\ldots,p_n$ be Laurent polynomials over $\mathbb{C}$ with supports $\mathcal{A}_1,\ldots,\mathcal{A}_n$. Then, the mixed volume ${\rm{MV}}(\mathcal{A}_1\cup\{\boldsymbol{0}\},\ldots,\mathcal{A}_n\cup\{\boldsymbol{0}\})$ is an upper bound for the number of common isolated zeros, counted with multiplicities, of the $p_i$ in $\mathbb{C}^n$.
\end{theorem}

An analysis of genericity conditions for solutions in $\mathbb{C}^n$ appears in \cite{Roj99} and an expository account of recent work in this area can be found in \cite{Roj03,DJD05}.

\section{Number of Limit Cycles Bifurcating from a Zero-Hopf \\ Equilibrium}\label{sect3-0}

\subsection{Bounds on the Number of Limit Cycles for Certain Generic Systems}\label{sub-sect3-1}
In order to apply the averaging method to study the zero-Hopf bifurcation of perturbed system \eqref{equ-3}, we need to perform several changes of coordinates (see \eqref{equA3-1} and \eqref{equA3-4}) to transform the perturbed system into the standard form of averaging. The approach used here is a straightforward generalization of the one proposed in \cite{BLV18}. The main results of this paper are based on the following lemma.
\begin{lemma}\label{lem-main-1}
The parametric formula of the standard form of averaging associated to system \eqref{equ-3} is as follows:
\begin{equation}\label{equ3-0-1}
\begin{split}
\frac{dR}{d\theta}&=\frac{R\cdot\sum_{i=1}^k\varepsilon^iS_{i,1}(\theta,R,X_3,\ldots,X_n)}{bR+\sum_{i=1}^k\varepsilon^iS_{i,2}(\theta,R,X_3,\ldots,X_n)},\\
&=\varepsilon F_{1,1}(\theta,R,X_3,\ldots,X_n)+\varepsilon^2F_{2,1}(\theta,R,X_3,\ldots,X_n)\\
&\quad+\cdots+\varepsilon^kF_{k,1}(\theta,R,X_3,\ldots,X_n)+\mathcal{O}(\varepsilon^{k+1}),\\
\frac{dX_s}{d\theta}&=\frac{R\cdot\sum_{i=1}^k\varepsilon^iS_{i,s}(\theta,R,X_3,\ldots,X_n)}{bR+\sum_{i=1}^k\varepsilon^iS_{i,2}(\theta,R,X_3,\ldots,X_n)},\\
&=\varepsilon F_{1,s}(\theta,R,X_3,\ldots,X_n)+\varepsilon^2F_{2,s}(\theta,R,X_3,\ldots,X_n)\\
&\quad+\cdots+\varepsilon^kF_{k,s}(\theta,R,X_3,\ldots,X_n)+\mathcal{O}(\varepsilon^{k+1}),
\quad s=3,\ldots,n,
\end{split}
\end{equation}
where
\begin{equation}\label{equ3-0-2}
\begin{split}
S_{i,1}&=a_iR+\sum(p_{1,i_1,\ldots,i_n,i-1}\cos\theta+p_{2,i_1,\ldots,i_n,i-1}\sin\theta)(R\cos\theta)^{i_1}(R\sin\theta)^{i_2}X_3^{i_3}\cdots X_n^{i_n},\\
S_{i,2}&=b_iR+\sum(p_{2,i_1,\ldots,i_n,i-1}\cos\theta-p_{1,i_1,\ldots,i_n,i-1}\sin\theta)(R\cos\theta)^{i_1}(R\sin\theta)^{i_2}X_3^{i_3}\cdots X_n^{i_n},\\
S_{i,s}&=c_{s,i}X_s+\sum p_{s,i_1,\ldots,i_n,i-1}(R\cos\theta)^{i_1}(R\sin\theta)^{i_2}X_3^{i_3}\cdots X_n^{i_n}\\
\end{split}
\end{equation}
with the sums in \eqref{equ3-0-2} ranging over $i_1+\cdots+i_n=m$, and the functions $F_{i,j}(\theta,R,X_3,\ldots,X_n)$, $i=1,2,\ldots,k$, $j=1,3,\ldots,n$, are obtained by carrying Taylor expansion of the expressions in \eqref{equ3-0-1} with respect to the variable $\varepsilon$ around $\varepsilon=0$.
\end{lemma}

\begin{proof}
Making the change of variables
\begin{equation}\label{equA3-1}
\begin{split}
x_1=r\cos\theta,\quad x_2=r\sin\theta,\quad x_s=x_s,\quad s=3,\ldots,n,
\end{split}
\end{equation}
with $r>0$, we see that system \eqref{equ-3} becomes
\begin{equation}\label{equA3-2}
\begin{split}
\frac{dr}{dt}&=\sum_{i=1}^k\varepsilon^i a_ir+\sum_{i=0}^{k-1}\varepsilon^i\sum(p_{1,i_1,i_2,\ldots,i_n,i}\cos\theta
+p_{2,i_1,i_2,\ldots,i_n,i}\sin\theta)(r\cos\theta)^{i_1}(r\sin\theta)^{i_2}x_3^{i_3}\cdots x_n^{i_n},\\
\frac{d\theta}{dt}&=\frac{1}{r}\Big((b+\sum_{i=1}^k\varepsilon^ib_i)r
+\sum_{i=0}^{k-1}\varepsilon^i\sum(p_{2,i_1,i_2,\ldots,i_n,i}\cos\theta-p_{1,i_1,i_2,\ldots,i_n,i}\sin\theta)(r\cos\theta)^{i_1}(r\sin\theta)^{i_2}x_3^{i_3}\cdots x_n^{i_n}\Big),\\
\frac{dx_s}{dt}&=\sum_{i=1}^{k}\varepsilon^ic_{s,i}x_s+\sum_{i=0}^{k-1}\varepsilon^i\sum p_{s,i_1,i_2,\ldots,i_n,i}(r\cos\theta)^{i_1}(r\sin\theta)^{i_2}x_3^{i_3}\cdots x_n^{i_n},\quad s=3,\ldots,n,
\end{split}
\end{equation}
where the sums go over $i_1+\cdots+i_n=m$. Since $b\neq0$, one can easily verify that in a suitable small neighborhood of $(r,x_3,\ldots,x_n)=(0,0,\ldots,0)$ with $r>0$ we always have $d \theta/dt\neq0$. Then taking $\theta$ as the new independent variable, in a neighborhood of $(r,x_3,\ldots,x_n)=(0,0,\ldots,0)$, one sees that system \eqref{equA3-2} becomes
\begin{equation}\label{equA3-3}
\begin{split}
\frac{dr}{d\theta}=\frac{dr/dt}{d\theta/dt},\quad \frac{dx_s}{d\theta}=\frac{dx_s/dt}{d\theta/dt}, \quad s=3,\ldots,n,
\end{split}
\end{equation}
where $dr/dt$, $d\theta/dt$ and $dx_s/dt$ take the expressions in \eqref{equA3-2}. To apply the averaging method, we rescale the variables by setting
\begin{equation}\label{equA3-4}
\begin{split}
(r,x_3,\ldots,x_n)=\left(\sqrt[m-1]{\varepsilon}R,\sqrt[m-1]{\varepsilon}X_3,\ldots,\sqrt[m-1]{\varepsilon}X_n\right),\quad 0<\varepsilon\ll1.
\end{split}
\end{equation}
Then system \eqref{equA3-3} becomes
\begin{equation}\label{equA3-5}
\begin{split}
\frac{dR}{d\theta}&=\frac{R\cdot\sum_{i=1}^k\varepsilon^iS_{i,1}(\theta,R,X_3,\ldots,X_n)}{bR+\sum_{i=1}^k\varepsilon^iS_{i,2}(\theta,R,X_3,\ldots,X_n)},\\
\frac{dX_s}{d\theta}&=\frac{R\cdot\sum_{i=1}^k\varepsilon^iS_{i,s}(\theta,R,X_3,\ldots,X_n)}{bR+\sum_{i=1}^k\varepsilon^iS_{i,2}(\theta,R,X_3,\ldots,X_n)},
\quad s=3,\ldots,n,
\end{split}
\end{equation}
where $S_{i,1}$, $S_{i,2}$ and $S_{i,s}$ for $s=3,\ldots,n$ are expressions in \eqref{equ3-0-2}. By carrying Taylor expansion of expressions in \eqref{equA3-5} with respect to the variable $\varepsilon$ around $\varepsilon=0$, one obtains the functions $F_{i,j}(\theta,R,X_3,\ldots,X_n)$, $i=1,2,\ldots,k$, $j=1,3,\ldots,n$, in \eqref{equ3-0-1}. This ends the proof of Lemma \ref{lem-main-1}.
\end{proof}

Denote the set of all the parameters other than $\varepsilon$ appearing in the perturbed system \eqref{equ-3} by $\boldsymbol{p}$, and $\boldsymbol{\eta}=(R,X_3,\ldots,X_n)\in\Omega$, where $\Omega$ is a suitable neighborhood of the origin. Let $\mathcal{R}[\boldsymbol{\eta}]$ be the real polynomial ring $\mathbb{R}(\boldsymbol{p})[\boldsymbol{\eta}]$. For system \eqref{equ3-0-1}, we have the following result on the $i$th-order averaged function $\boldsymbol{f}_i(\boldsymbol{\eta})$.

\begin{theorem}\label{main-th-0}
Let $\boldsymbol{f}_i(\boldsymbol{\eta})=(f_{i,1}(\boldsymbol{\eta}),f_{i,3}(\boldsymbol{\eta}),\ldots,f_{i,n}(\boldsymbol{\eta}))$ be the $i$th-order averaged function associated to system \eqref{equ3-0-1}, with $i\in\{1,2,\ldots,k\}$. Then for each $j\in\{1,3,\ldots,n\}$, there exists a smallest nonnegative integer $\mu_{i,j}\leq i-1$ such that $R^{\mu_{i,j}}f_{i,j}(\boldsymbol{\eta})=\bar{f}_{i,j}(\boldsymbol{\eta})\in\mathcal{R}[\boldsymbol{\eta}]$. Moreover, for each $j$, $\bar{f}_{i,j}(\boldsymbol{\eta})$ is a polynomial in $\mathcal{R}[\boldsymbol{\eta}]$ of total degree no more than $im$.
\end{theorem}

The proof of Theorem \ref{main-th-0} will be done by induction over $i$ and is given in Section \ref{sec-B}. The main idea behind the proof of Theorem \ref{main-th-0} which makes the induction procedure work is that we can derive the form of the function $F_{i,j}$ in \eqref{equ3-0-1} (see \eqref{equB-3}) and that the resulting denominator of the $L$-multilinear map in \eqref{equ2-4} contains only one variable. Note that the main work here is to study the maximum number of common real zeros of the obtained averaged functions. Since the transformation from the Cartesian coordinates $(x_1,x_2,\ldots,x_n)$ to the cylindrical coordinates $(r,\theta,x_3,\ldots,x_n)$ is not a diffeomorphism at $r=0$, we need to study the zeros of the averaged functions $\boldsymbol{f}_i(\boldsymbol{\eta})$ with $R>0$.


The following result is an immediate consequence of Theorem \ref{main-th-0}, which provides an upper bound for the number $H_k(n,m)$.

\begin{corollary}\label{main-co-1}
$H_k(n,m)\leq(km)^{n-1}$.
\end{corollary}
\begin{proof}
Let $\boldsymbol{f}_1(\boldsymbol{\eta})=\boldsymbol{f}_2(\boldsymbol{\eta})=\cdots=\boldsymbol{f}_{k-1}(\boldsymbol{\eta})\equiv0$.
By Theorem \ref{main-th-0}, there exists a smallest nonnegative integer $\mu_{k,j}\leq k-1$ such that $R^{\mu_{k,j}}f_{k,j}(\boldsymbol{\eta})=\bar{f}_{k,j}(\boldsymbol{\eta})\in\mathcal{R}[\boldsymbol{\eta}]$ for each $j\in\{1,3,\ldots,n\}$. Here $\bar{f}_{k,j}(\boldsymbol{\eta})$ is a polynomial in $\mathcal{R}[\boldsymbol{\eta}]$ of total degree at most $km$. By Bezout's theorem \cite{IRS74}, the maximum number of common isolated zeros that the algebraic system $\{\bar{f}_{k,1}(\boldsymbol{\eta}),\bar{f}_{k,3}(\boldsymbol{\eta}),\ldots,\bar{f}_{k,n}(\boldsymbol{\eta})\}$ can have is $(km)^{n-1}$. From Theorem \ref{averaging-thm}, we know that system \eqref{equ-3} up to the $k$th-order averaging has no more than $(km)^{n-1}$ limit cycles bifurcating from the origin at $\varepsilon=0$.
\end{proof}


Up to now and as far as we know, Corollary \ref{main-co-1} is the only result about zero-Hopf bifurcations in high-dimensional polynomial differential systems of arbitrary degree, obtained by using the averaging method of arbitrary order. To our knowledge, the result is the first one showing that the bound on the number of bifurcated limit cycles in a zero-Hopf bifurcation is a power function in the order $k$ and the degree $m$ of the systems. Some concrete applications of the averaging method to differential systems of lower degrees (or lower dimensions) can be seen in \cite{JLXZ09,BLV18,LMB09,LV11,BBM17,DBM21}, which shows that the upper bound given in Corollary \ref{main-co-1} is generally not reachable.

\subsection{Tight Bounds on the Number of Limit Cycles for Specialized Systems}\label{sub-sect3-2}
On the number of limit cycles of system \eqref{equ-3}, we have the following result obtained by using the first-order averaging method. Its proof can be found in Section \ref{sect-main-1}.

\begin{theorem}\label{main-th-1}
$H_1(n,m)\leq(m-1)\cdot m^{n-2}$.

\end{theorem}

Notice that, the upper bound in Theorem \ref{main-th-1} is smaller than the bound given in Corollary \ref{main-co-1} when $k=1$, because we can derive explicit expressions of the first-order averaged functions for generic $n$ and $m$ (see Section \ref{sect-main-1}). In general, the upper bound in Theorem \ref{main-th-1} is not optimal and can be improved (see Theorem \ref{main-th-2}).

We recall from \cite{JLXZ09,BLV18} that system \eqref{equ-3} can have $\ell$ limit cycles bifurcating from the origin with $\ell\in\{0,1,\ldots,2^{n-3}\}$ for $m=2$, $\ell\in\{0,1,\ldots,3^{n-2}\}$ for $m=3$, $\ell\leq6^{n-2}$ for $m=4$ (Theorem \ref{main-th-1} provides a sharper upper bound $3\cdot4^{n-2}$ when $n\geq5$), and $\ell\leq4\cdot5^{n-2}$ for $m=5$. The authors of \cite{PZ2009} provided a lower bound for $H_1(n,m)$ with general $m\geq2$, showing that either $H_1(n,m)\geq m^{n-2}/2$ when $m$ is even, or $H_1(n,m)\geq m^{n-2}$ when $m$ is odd.

The following corollary follows directly from Theorem \ref{main-th-1} and the results presented in \cite{JLXZ09,PZ2009,BLV18} and it gives lower and upper bounds for the number $H_1(n,m)$.

\begin{corollary}\label{cor-th-1}
The following statements hold for $H_1(n,m)$.
\begin{enumerate}
  \item[\emph{(a)}] $H_1(n,2)=2^{n-3}$.
  \item[\emph{(b)}] $H_1(n,3)=3^{n-2}$.
  \item[\emph{(c)}] $\min\{6^{n-2},3\cdot4^{n-2}\}\geq H_1(n,4)\geq2\cdot4^{n-3}$.
  \item[\emph{(d)}] $(m-1)\cdot m^{n-2}\geq H_1(n,m)\geq\left\{
\begin{array}{ll}
&m^{n-2}/2, \quad \text{when}\,\, m\geq5\,\, \text{is even}, \\
&m^{n-2}, \quad \text{when}\,\, m\geq5\,\, \text{is odd}.
\end{array}
\right.$
\end{enumerate}

\end{corollary}

The following result gives the number $H_1(n,m)$ for three-dimensional differential systems of the form \eqref{equ-3}.
\begin{theorem}\label{H-th-1}
For three-dimensional differential systems of the form \eqref{equ-3}, $H_1(3,m)=m/2$ for $m$ even, and $H_1(3,m)=m$ for $m$ odd.
\end{theorem}

\begin{proof}
The proof can be done by using the ideas from the proof of Theorem 23 in \cite{PZ2009}. Here we only present the proof for the case $m\geq2$ even; the case for $m$ odd can be proved similarly, see \cite{PZ2009} for more detailed arguments.

Noting that the coefficients of the function $\bar{f}_{1,1}(\boldsymbol{\eta})$ in \eqref{equ3-9} are arbitrary, one can simplify the function by writing it in the form
\begin{equation}\label{equ6-1}
\begin{split}
\bar{f}_{1,1}(\boldsymbol{\eta})&=\bar{A}+\bar{A}_1R^{m-2}X_3
+\bar{A}_3R^{m-4}X_3^3+\cdots+\bar{A}_{m-1}X_3^{m-1},
\end{split}
\end{equation}
where $\bar{A}$ and $\bar{A}_i$ are arbitrary nonzero constants. In a similar way, the function $\bar{f}_{1,3}(\boldsymbol{\eta})$ (see \eqref{equ3-10}) can be simplified to
\begin{equation}\label{equ6-2}
\begin{split}
\bar{f}_{1,3}(\boldsymbol{\eta})&=\bar{B}X_3+\bar{B}_0R^{m}
+\bar{B}_2R^{m-2}X_3^2+\cdots+\bar{B}_{m}X_3^{m},
\end{split}
\end{equation}
where $\bar{B}$ and $\bar{B}_i$ are arbitrary nonzero constants. Since $R\neq0$, it follows from \eqref{equ6-1} and \eqref{equ6-2} that
\begin{equation}\label{equ6-3}
\begin{split}
\frac{\bar{B}X_3}{\bar{A}R}&=\frac{\bar{B}_0R^{m}
+\bar{B}_2R^{m-2}X_3^2+\cdots+\bar{B}_{m}X_3^{m}}{R(\bar{A}_1R^{m-2}X_3
+\bar{A}_3R^{m-4}X_3^3+\cdots+\bar{A}_{m-1}X_3^{m-1})}\\
&=\frac{\bar{B}_0
+\bar{B}_2\left(\frac{X_3}{R}\right)^2+\cdots+\bar{B}_{m}\left(\frac{X_3}{R}\right)^m}
{\bar{A}_1\left(\frac{X_3}{R}\right)+\bar{A}_3\left(\frac{X_3}{R}\right)^3
+\cdots+\bar{A}_{m-1}\left(\frac{X_3}{R}\right)^{m-1}}.
\end{split}
\end{equation}
Consequently, we have
\begin{equation}\label{equ6-4}
\begin{split}
\bar{B}_0+\left(\bar{B}_2-\frac{\bar{B}\bar{A}_1}{\bar{A}}\right)\left(\frac{X_3}{R}\right)^2
+\left(\bar{B}_m-\frac{\bar{B}\bar{A}_{m-1}}{\bar{A}}\right)\left(\frac{X_3}{R}\right)^m=0.
\end{split}
\end{equation}
Since the coefficients of the equation \eqref{equ6-4} in $X_3/R$ can be arbitrary real numbers, so for an appropriate choice of the values of the coefficients, the equation \eqref{equ6-4} can have exactly $m/2$ simple nonzero real solutions $(X_3/R)^2$.

Note that the equation \eqref{equ6-1} can be rewritten as
\begin{equation}\label{equ6-5}
\begin{split}
X_3^{m-1}=\frac{-\bar{A}}{\bar{A}_{m-1}+\bar{A}_{m-3}\left(\frac{R}{X_3}\right)^2
+\cdots+\bar{A}_1\left(\frac{R}{X_3}\right)^{m-2}}.
\end{split}
\end{equation}
Since $m$ is even, for each solution $X_3/R$ of \eqref{equ6-4} we obtain a unique
real solution
\begin{equation}\label{equ6-6}
\begin{split}
X_3=\left(\frac{-\bar{A}}{\bar{A}_{m-1}+\bar{A}_{m-3}\left(\frac{R}{X_3}\right)^2
+\cdots+\bar{A}_1\left(\frac{R}{X_3}\right)^{m-2}}\right)^{\frac{1}{m-1}}.
\end{split}
\end{equation}
Hence, the algebraic system $\{\bar{f}_{1,1}(\boldsymbol{\eta}),\bar{f}_{1,3}(\boldsymbol{\eta})\}$ has exactly $m/2$ simple zeros, denoted by $(R_i,X_{3,i})$ with $R_i>0$ for $i=1,2,\ldots,m/2$. This completes the proof.

\end{proof}

\begin{table}[h]
\begin{center}
\caption{Some values of $H_1(n,m)$ for system \eqref{equ-3}.}\label{Tab-1}
\begin{tabular}{|c|c|c||c||c||c||c||c||c|}
  \hline
  \multicolumn{9}{|c|}{$n$}  \\
  \hline \hline
  & & 3 & 4 & 5 & 6 & 7 & 8 & $\cdots$\\
  \hline \hline
  \multirow{8}{*}{$m$}
  & 2 & 1 &  2 & 4 & 8 & 16 & 32 & $\rightarrow$\\
  & 3 & 3 &  9 & 27 & 81 & 243 & 729 &  $\rightarrow$\\
  & 4 & 2 &  &  &  &  &  &  \\
  & 5 & 5 &  &  &  &  &  &  \\
  & 6 & $3^*$ &  &  &  &  &  & \\
  & 7 & $7^*$ &  &  &  &  &  &  \\
  & 8 & 4 &  &  &  &  &  & \\
  & $\vdots$ & $\downarrow$ &  &  &  &  &  & \\
  \hline
\end{tabular}
\end{center}
\end{table}

Theorem \ref{H-th-1} shows that $H_1(n,m)$ equals to its lower bound when $n=3$. We list in Table \ref{Tab-1} the known values of $H_1(n,m)$ for system \eqref{equ-3}. In the three-dimensional case, we succeeded only in finding the values of $H_1(3,6)$ and $H_1(3,7)$ by means of Gr\"obner basis (marked with asterisks in Table \ref{Tab-1}). In a similar way, we once more confirmed the results that were obtained in \cite{BLV18}. In the more difficult case of $m=8$, we found out that the computation of the Gr\"obner basis of the resulting polynomial system in Maple was consuming too much of the CPU. Surely, in near future, the progress in computers will advance the study of the class of systems in question, and far beyond the point we have reached.

In the following, we avoid computing Gr\"obner basis and using the method of \textit{mixed volume} and Bernstein's theorem to predict an upper bound of $H_1(n,m)$ that is more refined than the bound provided in Theorem \ref{main-th-1}. As is well known, Bernstein's bound is at most as high as B\'ezout's bound, which is simply the product of the total degrees, and is usually significantly smaller for systems encountered in real-world applications (especially for sparse polynomial systems).

\begin{theorem}\label{main-th-2}
The number $H_1(n,m)$ is bounded by the mixed volume ${\rm{MV}}(\mathcal{A}_1,\{\boldsymbol{0}_{n-1}\}\cup\mathcal{A}_3,\ldots,\{\boldsymbol{0}_{n-1}\}\cup\mathcal{A}_n)$, where $\mathcal{A}_s$ is the support of the polynomial $\bar{f}_{1,s}(\boldsymbol{\eta})$ (see Section \ref{sect-main-1}), with $s\in\{1,3,\ldots,n\}$. Moreover,
\begin{equation}\label{equmx-1}
\begin{split}
\mathcal{A}_1&=\mathop{\bigcup}\limits_{t=1}^{{m}/{2}}\{({m}/{2}-t,\boldsymbol{e}_{j_1j_2\cdots j_{2t-1}})\}\cup\{\boldsymbol{0}_{n-1}\},\\
\mathcal{A}_s&=\mathop{\bigcup}\limits_{t=1}^{{m}/{2}}\{({m}/{2}-t,\boldsymbol{e}_{j_1j_2\cdots j_{2t}})\}\cup\{\boldsymbol{\bar{e}}_{s}\}
\cup\{(m/2,\boldsymbol{0}_{n-2})\},\quad s=3,\ldots,n
\end{split}
\end{equation}
when $m\geq2$ is even, and
\begin{equation}\label{equmx-2}
\begin{split}
\mathcal{A}_1&=\mathop{\bigcup}\limits_{t=1}^{{(m-1)}/{2}}\{({(m-1)}/{2}-t,\boldsymbol{e}_{j_1j_2\cdots j_{2t}})\}\cup\{\boldsymbol{0}_{n-1}\}\cup\{((m-1)/2,\boldsymbol{0}_{n-2})\},\\
\mathcal{A}_s&=\mathop{\bigcup}\limits_{t=1}^{{(m+1)}/{2}}\{({(m+1)}/{2}-t,\boldsymbol{e}_{j_1j_2\cdots j_{2t-1}})\}\cup\{\boldsymbol{\bar{e}}_{s}\},\quad s=3,\ldots,n
\end{split}
\end{equation}
when $m\geq3$ is odd, where $\boldsymbol{\bar{e}}_{s}\in\mathbb{N}^{n-1}$ is the unit vector with $(s-1)$th entry equal to 1, and $\boldsymbol{e}_{j_1j_2\cdots j_t}\in\mathbb{N}^{n-2}$ has the sum of the $j_1$th, the $j_2$th, $\ldots$, and the $j_t$th entries equal to $t$ and the others equal to 0 (these entries can coincide).
\end{theorem}

\begin{proof}
We want to bound the number of common isolated zeros of the algebraic system $\{\bar{f}_{1,1}(\boldsymbol{\eta}),\bar{f}_{1,3}(\boldsymbol{\eta})$, $\ldots,\bar{f}_{1,n}(\boldsymbol{\eta})\}$ in $\mathbb{C}^{n-1}$. It follows from Theorem \ref{LW-thm} that $H_1(n,m)$ is bounded by the mixed volume ${\rm{MV}}(\mathcal{A}_1,\{\boldsymbol{0}_{n-1}\}\cup\mathcal{A}_3,\ldots,\{\boldsymbol{0}_{n-1}\}\cup\mathcal{A}_n)$. Let $\rho=R^2$ be a new variable in the polynomials $\bar{f}_{1,j}(\boldsymbol{\eta})$ for $j=1,3,\ldots,n$. Then when $m$ is even, the support $\mathcal{A}_s$ of $\bar{f}_{1,s}(\boldsymbol{\eta})$ in the variables $(\rho,X_3,\ldots,X_n)$ can be derived from the equations \eqref{equ3-9} and \eqref{equ3-10} respectively for $s=1,3,\ldots,n$. This completes the proof of formula \eqref{equmx-1}. The discussions for \eqref{equmx-2} are closely similar, so we omit them.
\end{proof}

\begin{remark}\label{rem-th-2}
Recall that the set of homogeneous polynomials in $n$ variables of degree $d$ can be associated with a vector space of dimension $\binom{n+d-1}{d}$. By Theorem \ref{main-th-2}, we have
\begin{equation}\label{equmx-3}
\begin{split}
|\mathcal{A}_1|=1+\sum_{t=1}^{m/2}\binom{n+2t-4}{2t-1},\quad
|\mathcal{A}_s|=2+\sum_{t=1}^{m/2}\binom{n+2t-3}{2t},\quad s=3,\ldots,n\nonumber
\end{split}
\end{equation}
when $m$ is even, and
\begin{equation}\label{equmx-4}
\begin{split}
|\mathcal{A}_1|=2+\sum_{t=1}^{(m-1)/2}\binom{n+2t-3}{2t},\quad
|\mathcal{A}_s|=1+\sum_{t=1}^{(m+1)/2}\binom{n+2t-4}{2t-1},\quad s=3,\ldots,n\nonumber
\end{split}
\end{equation}
when $m$ is odd.
\end{remark}

In Table \ref{Tab-2}, we present some values of the BKK bound given in Theorem \ref{main-th-2}. The results on the mixed volume were computed by using Emiris and Canny's algorithm \cite{EC95}, for which the software is available at https://github.com/iemiris/MixedVolume-SparseResultants. There are two other software packages, {\sf MixedVol-2.0} \cite{LeeL11} and {\sf PHCpack} in Macaulay2 \cite{GPV13}, which can be used for computing mixed volumes of polynomial systems.

\begin{table}[h]
\begin{center}
\caption{Some BKK bounds of $H_1(n,m)$.}\label{Tab-2}
\begin{tabular}{|c|c|c||c||c||c||c||c||c|}
  \hline
  \multicolumn{9}{|c|}{$n$}  \\
  \hline \hline
  & & 3 & 4 & 5 & 6 & 7 & 8 & $\cdots$\\
  \hline \hline
  \multirow{8}{*}{$m$}
  & 2 & {\color{red}{1}} &  {\color{red}{2}} & {\color{red}{4}} & {\color{red}{8}} & {\color{red}{16}} & {\color{red}{32}} &  $\rightarrow$\\
  & 3 & {\color{red}{3}} &  {\color{red}{9}} & {\color{red}{27}} & {\color{red}{81}} & {\color{red}{243}} & {\color{red}{729}} &  $\rightarrow$\\
  & 4 & 6 &  24 & 96 & 384 & 1536 & - &  \\
  & 5 & 10 & 50 & 250 & 1250 & - & - &  \\
  & 6 & 15 & 90 & 540 &  - & - & - &  \\
  & 7 & 21 & 147 & 1029 & - & - & - &  \\
  & 8 & 28 & 224 & 1792 & - & - & - &  \\
  & $\vdots$ &  &  &  &  &  &  & $\ddots$\\
  \hline
\end{tabular}
\end{center}
\end{table}

The mixed volume program we used may report an error when $n$ or $m$ becomes bigger and the resulting matrix for Linear Programming (LP) is larger than expected. Through observation we obtain the following corollary of Theorem \ref{main-th-2} (the numbers in red are the exact values of $H_1(n,m)$).

\begin{corollary}\label{BKK-cor}
The following equalities hold for $m=2$ and $m=3$:
\begin{equation}\label{equmx-5}
{\rm{MV}}(\mathcal{A}_1,\{\boldsymbol{0}_{n-1}\}\cup\mathcal{A}_3,\ldots,\{\boldsymbol{0}_{n-1}\}\cup\mathcal{A}_n)=
\left\{
\begin{array}{ll}
&H_1(n,2)=2^{n-3},\\
&H_1(n,3)=3^{n-2}.
\end{array}
\right.
\end{equation}

\end{corollary}

\begin{proof}
The desired result for $m=2$ follows from Theorem 1 (or Corollary 2) of \cite{JLXZ09} and the result for $m=3$ follows from Theorem 1 of \cite{BLV18}. In each case, the maximum number of real isolated zeros of the resulting algebraic system $\{\bar{f}_{1,1}(\boldsymbol{\eta}),\bar{f}_{1,3}(\boldsymbol{\eta}),\ldots,\bar{f}_{1,n}(\boldsymbol{\eta})\}$
equals to its B\'ezout's bound.
\end{proof}

\begin{table}[h]
\begin{center}
\caption{Some BKK bounds of $H_2(n,m)$.}\label{Tab-3}
\begin{tabular}{|c|c|c||c||c||c||c||c||c|}
  \hline
  \multicolumn{9}{|c|}{$n$}  \\
  \hline \hline
  & & 3 & 4 & 5 & 6 & 7 & 8 & $\cdots$\\
  \hline \hline
  \multirow{8}{*}{$m$}
  & 2 & {\color{red}{3}} &  {\color{red}{9}} & 27 & 81 & 243 & 729 &  $\cdots$\\
  & 3 & 10 & 50 & 250 & 1250 & - & - &  $\cdots$\\
  & 4 & 21 &  147 & 1029 & - & - & - &  \\
  & 5 & 36 & 324 & - & - & - & - &  \\
  & 6 & 55 & 605 & - &  - & - & - &  \\
  & 7 & 78 & - & - & - & - & - &  \\
  & 8 & 105 & - & - & - & - & - &  \\
  & $\vdots$ &  &  &  &  &  &  & $\ddots$\\
  \hline
\end{tabular}
\end{center}
\end{table}

We provide in Table \ref{Tab-3} some values of the BKK bound of $H_2(n,m)$. Recently, the result $H_2(4,2)=9$ is proved by Djedid and others \cite[Theorem 1]{DBM21}. Note that the second-order averaged functions were computed by using
the conditions of $\bar{f}_{1,1}(\boldsymbol{\eta})=\bar{f}_{1,3}(\boldsymbol{\eta})=\cdots=\bar{f}_{1,n}(\boldsymbol{\eta})\equiv0$. By observing the numbers of the BKK bound, we pose the following conjecture.

\begin{conjecture}\label{conj-1}
Let $\nu_{n,m}$ denote the BKK bound of $H_2(n,m)$. Then $\nu_{n,m}=\nu_{3,m}\times(2m-1)^{n-3}$.
\end{conjecture}

Recall that what we are mainly interested in is how many of the complex isolated zeros (provided by the BKK bound for a given class of polynomial systems) are real? The answer to this question depends very much on the type of the systems. In general, it is difficult to estimate the number of real solutions of a polynomial system unless the system has special structures. For work related to this question, we refer the reader to \cite{BS94,BS11,FS11}. Concerning the algebraic systems we encountered in the study of limit cycles, we have not found any bound for the total number of real solutions that is smaller than the BKK bound. How to find such a small bound is still one of the main problems that remain for research.

\subsection{Algorithmic Derivation of Conditions for a Specified Number of Limit Cycles}\label{sub-sect3-3}

Our third question, relevant to the algorithmic aspects of limit-cycle bifurcation, then arises: \textit{under what conditions does a differential system of the form \eqref{equ-3} have a prescribed number of limit cycles}? In the following we provide an algorithmic approach to answer the question by means of symbolic and algebraic computation. Let $\boldsymbol{f}_k(\boldsymbol{\eta})=(f_{k,1}(\boldsymbol{\eta}),f_{k,3}(\boldsymbol{\eta}),\ldots,f_{k,n}(\boldsymbol{\eta}))$ be the $k$th-order averaged function associated to system \eqref{equ-3}. Denote the Jacobian of the function $\boldsymbol{f}_k(\boldsymbol{\eta})$ by $J_{\boldsymbol{f}_k}(\boldsymbol{\eta})$. That is,
\begin{equation}\label{equ3-0-3}
\begin{split}
J_{\boldsymbol{f}_k}(\boldsymbol{\eta})=
\left[
 \begin{matrix}
   \frac{\partial f_{k,1}}{\partial R} & \frac{\partial f_{k,1}}{\partial X_3} & \cdots& \frac{\partial f_{k,1}}{\partial X_n} \\
   \frac{\partial f_{k,3}}{\partial R} & \frac{\partial f_{k,3}}{\partial X_3} & \cdots& \frac{\partial f_{k,3}}{\partial X_n} \\
   \vdots& \vdots & \ddots & \vdots \\
   \frac{\partial f_{k,n}}{\partial R} & \frac{\partial f_{k,n}}{\partial X_3} & \cdots& \frac{\partial f_{k,n}}{\partial X_n} \\
  \end{matrix}
  \right].\nonumber
\end{split}
\end{equation}

Let $D_k(\boldsymbol{\eta})$ be the determinant of the Jacobian $J_{\boldsymbol{f}_k}(\boldsymbol{\eta})$. The following theorem provides sufficient conditions for system \eqref{equ-3} to have exactly $\ell$ limit cycles that can bifurcate from the origin.

\begin{theorem}\label{semi-averaging}
For $\varepsilon>0$ sufficiently small, system \eqref{equ-3} up to the $k$th-order averaging has exactly $\ell$ limit cycles that can bifurcate from the origin if the following semi-algebraic system
\begin{equation}\label{equ3-0-4}
\begin{split}
\left\{
\begin{array}{ll}
&\bar{f}_{k,1}(\boldsymbol{\eta})=\bar{f}_{k,3}(\boldsymbol{\eta})
=\cdots=\bar{f}_{k,n}(\boldsymbol{\eta})=0, \\
&R>0,\quad \bar{D}_{k}(\boldsymbol{\eta})\neq0, \quad b\neq0
\end{array}
\right.
\end{split}
\end{equation}
has exactly $\ell$ distinct real solutions with respective to the variables $\boldsymbol{\eta}$. Here $\bar{f}_{k,j}(\boldsymbol{\eta})$, for $j=1,3,\ldots,n$, are polynomials obtained in Theorem \ref{main-th-0} and $\bar{D}_{k}(\boldsymbol{\eta})$ is the numerator of the determinant $D_k(\boldsymbol{\eta})$.
\end{theorem}

\begin{proof}
Assume that system \eqref{equ3-0-4} has $\ell$ distinct real solutions, which can be written as $\boldsymbol{\alpha}_1=(R^{(1)},X_3^{(1)},\ldots,X_n^{(1)})$, $\boldsymbol{\alpha}_2=(R^{(2)},X_3^{(2)},\ldots,X_n^{(2)}), \ldots, \boldsymbol{\alpha}_{\ell}=(R^{(\ell)},X_3^{(\ell)},\ldots,X_n^{(\ell)})$. Note that, for each solution $\boldsymbol{\alpha}_j$, $\bar{D}_{k}(\boldsymbol{\alpha}_j)\neq0$ implies that $D_k(\boldsymbol{\eta})\neq0$ in a suitable small neighborhood of $\boldsymbol{\eta}=(0,0,\ldots,0)$ with $R>0$. From Theorem \ref{averaging-thm}, we know that for $\varepsilon>0$ sufficiently small, system \eqref{equ-3} has a $T$-periodic solution $\boldsymbol{x}(t,\boldsymbol{\alpha}_j,\varepsilon)$ such that $\boldsymbol{x}(0,\boldsymbol{\alpha}_j,\varepsilon)\rightarrow\boldsymbol{\alpha}_j$ when $\varepsilon\rightarrow0$. This completes the proof.
\end{proof}

Theorem \ref{semi-averaging} provides an effective and straightforward computational method to verify whether an obtained bound ($\ell$) for the number of limit cycles of a given differential system can be reached. Its main task is to find the conditions on the parameters for system \eqref{equ3-0-4} to have exactly $\ell$ distinct real solutions. There are several software packages that can be used for solving such semi-algebraic systems. These packages implement the method of discriminant varieties of Lazard and Rouillier \cite{DLFR} and the method of Yang and Xia \cite{LYBX05} for real solution classification.
In practice, we first compute the mixed volume of the polynomial
system $\{\bar{f}_{k,1}(\boldsymbol{\eta}),\bar{f}_{k,3}(\boldsymbol{\eta}),\ldots,\bar{f}_{k,n}(\boldsymbol{\eta})\}$ to obtain an upper bound for the number of common isolated zeros of the system and then use Theorem \ref{semi-averaging} to check whether the obtained bound is reached. In Sections \ref{sect-diff} and \ref{sec-hyper}, we will demonstrate our algorithmic tests using a class of third-order differential equations and a four-dimensional hyperchaotic differential system.

\section{Existence of Periodic Solutions of Third-Order Differential \\ Equations}\label{sect-diff}
As an application of our main results, we now study the existence of periodic solutions of the following third-order differential equation:
\begin{equation}\label{eqap-1}
\begin{split}
\frac{d^3x}{dt^3}+p\frac{d^2x}{dt^2}+q\frac{dx}{dt}+\ell x=f\Big(x,\frac{dx}{dt},\frac{d^2x}{dt^2}\Big),
\end{split}
\end{equation}
where $p$, $q$ and $\ell$ are real parameters, and
\begin{equation}\label{eqap-2}
\begin{split}
f\Big(x,\frac{dx}{dt},\frac{d^2x}{dt^2}\Big)=a_1x^2+a_2x\frac{dx}{dt}+a_3x\frac{d^2x}{dt^2}
+a_4\Big(\frac{dx}{dt}\Big)^2+a_5\frac{dx}{dt}\frac{d^2x}{dt^2}
+a_6\Big(\frac{d^2x}{dt^2}\Big)^2+\mathcal{O}(3)\nonumber
\end{split}
\end{equation}
is a $\mathcal{C}^3$ differential function. Set $x_1=x$, $x_2=\dot{x}_1$ and $x_3=\dot{x}_2$, where the dot denotes derivative with respect to the variable $t$. Then we can write equation \eqref{eqap-1} as the following differential system
\begin{equation}\label{eqap-3}
\begin{split}
\dot{x}_1&=x_2,\quad \dot{x}_2=x_3,\\
\dot{x}_3&=-\ell x_1-qx_2-px_3+f(x_1,x_2,x_3).
\end{split}
\end{equation}
We want to analyze the existence of limit cycles of system \eqref{eqap-3} by using the second-order averaging method. In order to transform system \eqref{eqap-3} into the form \eqref{equ-3}, we assume that
\begin{equation}\label{eqap-4}
\begin{split}
\ell&=-\varepsilon(\varepsilon\mu_2+\mu_1)(\varepsilon^4\alpha_2^2
+2\varepsilon^{3}\alpha_{{1}}\alpha_{{2}}+\varepsilon^{2}\alpha_{{1}}^{2}+\beta^{2}),\\
p&=-\varepsilon(\varepsilon\mu_{{2}}+2\varepsilon\alpha_{{2}}+\mu_{{1}}
+2\alpha_{{1}}),\\
q&=\beta^2+\varepsilon^2(\varepsilon\alpha_2+\alpha_1)
(2\varepsilon\mu_{{2}}+\varepsilon\alpha_{{2}}+2\mu_{{1}}+\alpha_{{1}}),\\
a_i&=a_{i,0}+\varepsilon a_{i,1}+\varepsilon^2a_{i,2},\quad i=1,2,\ldots,6,
\end{split}
\end{equation}
where the $\alpha_j$'s, $\mu_j$'s, and $a_{i,j}$'s are real parameters. Applying the change of variables
\begin{equation}\label{eqap-5}
\begin{split}
(x,y,z)^T=M(x_1,x_2,x_3)^T,
\end{split}
\end{equation}
with
\begin{equation}\label{eqap-6}
\begin{split}
M=\left[
 \begin{matrix}
  \varepsilon^2(\varepsilon\mu_2+\mu_1)(\varepsilon\alpha_2+\alpha_1)  &  -\varepsilon(\varepsilon\mu_{{2}}+\varepsilon\alpha_{{2}}
    +\mu_{{1}}+\alpha_{{1}}) & 1 \\
   -\beta\varepsilon(\varepsilon\mu_2+\mu_1) & \beta & 0 \\
  \beta^2+\varepsilon^2(\varepsilon\alpha_2+\alpha_1)^2  & -2\varepsilon(\varepsilon\alpha_2+\alpha_1) & 1 \\
  \end{matrix}
  \right],\nonumber
\end{split}
\end{equation}
one sees that system \eqref{eqap-3} becomes
\begin{equation}\label{eqap-7}
\begin{split}
\dot{x}&=(\varepsilon\alpha_1+\varepsilon^2\alpha_2)x-\beta y+ g(x,y,z,\varepsilon),\\
\dot{y}&=\beta x+(\varepsilon\alpha_1+\varepsilon^2\alpha_2)y,\\
\dot{z}&=(\varepsilon\mu_1+\varepsilon^2\mu_2)z+g(x,y,z,\varepsilon).
\end{split}
\end{equation}
The expression of $g(x,y,z,\varepsilon)$ is extremely long and we do not produce it
here. Our main result on the limit cycles of system \eqref{eqap-7} is the following.
\begin{theorem}\label{main-th-3}
The following statements hold for $\varepsilon>0$ sufficiently small.
\begin{itemize}
\item [\emph{(a)}] System \eqref{eqap-7} has, up to the first-order averaging, at most 1 limit cycle bifurcating from the origin, and this number can be reached if one of the following 4 conditions holds:
\begin{equation}\label{eqap-8}
\begin{split}
\mathcal{C}_1&=[R_1,\,R_2<0;\,\,R_3>0],\quad \mathcal{C}_2=[R_1,\,R_3<0;\,\,R_2>0],\\
\mathcal{C}_3&=[R_2,\,R_3<0;\,\,R_1>0],\quad \mathcal{C}_4=[R_1,\,R_2,\,R_3>0],
\end{split}
\end{equation}
where
\begin{equation}\label{eqap-9}
\begin{split}
R_1&=\alpha_1,\quad
R_2=a_{{6,0}}\beta^{4}+(a_{{4,0}}-a_{{3,0}})\beta^{2}+a_{{1,0}},\\
R_3&=a_{{3,0}}\mu_{{1}}\beta^{2}-2a_{{1,0}}\mu_{{1}}-2a_{{1,0}}\alpha_{{1}}.\nonumber
\end{split}
\end{equation}

\item [\emph{(b)}] System \eqref{eqap-7} has, up to the second-order averaging, at most 2 limit cycles bifurcating from the origin, and this number can be reached if one of the following 4 conditions holds:
\begin{equation}\label{eqap-10}
\begin{split}
\bar{\mathcal{C}}_1&=[\bar{R}_1,\,\bar{R}_2,\,\bar{R}_3<0;\,\,\bar{R}_4\geq0;\,\,\bar{R}_5>0],\\
\bar{\mathcal{C}}_2&=[\bar{R}_1<0;\,\,\bar{R}_4\leq0;\,\,\bar{R}_2,\,\bar{R}_3,\,\bar{R}_5>0],\\
\bar{\mathcal{C}}_3&=[\bar{R}_2<0;\,\,\bar{R}_4\geq0;\,\,\bar{R}_1,\,\bar{R}_3,\,\bar{R}_5>0],\\
\bar{\mathcal{C}}_4&=[\bar{R}_3<0;\,\,\bar{R}_4\leq0;\,\,\bar{R}_1,\,\bar{R}_2,\,\bar{R}_5>0],
\end{split}
\end{equation}
where
\begin{equation}\label{eqap-11}
\begin{split}
\bar{R}_1&=\alpha_2,\quad \bar{R}_2=2a_{{3,1}}\beta^{2}-3a_{{1,1}},\quad
\bar{R}_3=a_{{3,1}}\mu_{{2}}\beta^{2}-2a_{{1,1}}\mu_{{2}}-2a_{{1,1}}\alpha_{{2}},\\
\bar{R}_4&=2a_{{6,1}}a_{{3,1}}^{2}\beta^{8}+\big(4a_{{2,0}}a_{{3,1}}a_{{6,0}}\alpha_{{2}}
-a_{{2,0}}a_{{3,1}}a_{{6,0}}\mu_{{2}}-8a_{{1,1}}
a_{{3,1}}a_{{6,1}}-2a_{{3,1}}^{3}+2a_{{4,1}}a_{{3,1}}^{2}\big)\beta^{6}\\
&\quad+\big(2a_{{2,0}}a_{{1,1}}a_{{6,0}}\mu_{{2}}-4
a_{{2,0}}a_{{1,1}}a_{{6,0}}\alpha_{{2}}+8a_{{1,1}}^{2}a_{{6,1}}+
10a_{{1,1}}a_{{3,1}}^{2}-8a_{{1,1}}a_{{4,1}}a_{{3,1}}\big)\beta^{4}\\
&\quad+\big(8a_{{1,1}}^{2}a_{{4,1}}-16a_{{1,1}}^{2}a_{{3,1}}\big)\beta^{2}+8a_{{1,1}}^{3},\\
\bar{R}_5&=4a_{{3,1}}^{2}a_{{6,1}}^{2}\beta^{12}+\big(16a_{{2,0}}a_{{3,1}}a_{{6,0}}a_{{6
,1}}\alpha_{{2}}-4a_{{2,0}}a_{{3,1}}a_{{6,0}}a_{{6,1}}\mu_{{2}}-16a_{{1,1}}a_{{3,1}}a_{{6,1}}^{2}-8a_{{3,1}}^{3}a_{{6,1}}\\
&\quad+8a_{{3,1}}^{2}a_{{4,1}}a_{{6,1}}\big)\beta^{10}
+\big(a_{{2,0}}^{2}a_{{6,0}}^{2}\mu_{{2}}^{2}+8a_{{2,0}}^
{2}a_{{6,0}}^{2}\mu_{{2}}\alpha_{{2}}+16a_{{2,0}}^{2}a_{{6,0}}
^{2}\alpha_{{2}}^{2}+8a_{{1,1}}a_{{2,0}}a_{{6,0}}a_{{6,1}}\mu_{{2}}\\
&\quad-16a_{{1,1}}a_{{2,0}}a_{{6,0}}a_{{6,1}}\alpha_{{2}}+4a_{{3,1}}^
{2}\mu_{{2}}a_{{2,0}}a_{{6,0}}-16a_{{3,1}}^{2}\alpha_{{2}}a_{{2,0}
}a_{{6,0}}-4a_{{2,0}}a_{{3,1}}a_{{4,1}}a_{{6,0}}\mu_{{2}}\\
&\quad+16a_{{1,1}}^{2}a_{{6,1}}^{2}+16a_{{2,0}}a_{{3,1}}a_{{4,1}}a_{{6,0}}\alpha_{{2}}+40a_{{1,1}}a_{{3,1}}^{2}a_{{6,1}}-32a_{{1,1}}a_{{3,1}}a_{
{4,1}}a_{{6,1}}+4a_{{3,1}}^{4}-8a_{{3,1}}^{3}a_{{4,1}}\\
&\quad+4a_{{3,1}}^{2}a_{{4,1}}^{2}\big)\beta^{8}
+\big(32a_{{1,1}}a_{{3,1}}\alpha_{{2}}a_{{2,0}}a_{{6,0}}-12a_{{1,1}}a
_{{3,1}}\mu_{{2}}a_{{2,0}}a_{{6,0}}+8a_{{1,1}}a_{{2,0}}a_{{4,1}}a_{{6,0}}\mu_{{2}}\\
&\quad-16a_{{1,1}}a_{{2,0}}a_{{4,1}}a_{{6,0}}\alpha_{{2}}
-64a_{{1,1}}^{2}a_{{3,1}}a_{{6,1}}
+32a_{{1,1}}^{2}a_{{4,1}}a_{{6,1}}-24a_{{1,1}}a_{{3,1}}^{3}
+40a_{{1,1}}a_{{3,1}}^{2}a_{{4,1}}\\
&\quad-16a_{{1,1}}a_{{3,1}}a_{{4,1}}^{2}\big)\beta^{6}
+\big(8a_{{1,1}}^{2}\mu_{{2}}a_{{2,0}}a_{{6,0}}-16a_{{1,1}}^{2}\alpha_{{2}}a_{{2,0}}a_{{6,0}}
+32a_{{1,1}}^{3}a_{{6,1}}+52a_{{1,1}}^{2}a_{{3,1}}^{2}\\
&\quad-64a_{{1,1}}^{2}a_{{3,1}}a_{{4,1}}+16a_{{1,1}}^{2}a_{{4,1}}^{2}\big)\beta^{4}
+\big(32a_{{1,1}}^{3}a_{{4,1}}-48a_{{1,1}}^{3}a_{{3,1}}\big)\beta^{2}+16a_{{1,1}}^{4}.\nonumber
\end{split}
\end{equation}

\end{itemize}
\end{theorem}

\begin{proof}
Computing the third-order Taylor expansion of the expressions in \eqref{eqap-7}, with respect to $\varepsilon$ around $\varepsilon=0$, we obtain
\begin{equation}\label{eq-diff-1}
\begin{split}
\dot{x}&=(\varepsilon\alpha_1+\varepsilon^2\alpha_2)x-\beta y+ g_0(x,y,z)+\varepsilon g_1(x,y,z),\\
\dot{y}&=\beta x+(\varepsilon\alpha_1+\varepsilon^2\alpha_2)y,\\
\dot{z}&=(\varepsilon\mu_1+\varepsilon^2\mu_2)z+g_0(x,y,z)+\varepsilon g_1(x,y,z),\\
\end{split}
\end{equation}
where
\begin{small}
\begin{equation}\label{eq-diff-2}
\begin{split}
g_0(x,y,z)&=\frac{1}{\beta^4}\Big((\beta^{3}a_{{5,0}}-\beta a_{{2,0}})xy+\beta a_{{2,0}}yz+(\beta^{2}a_{{3,0}}-2a_{{1,0}})xz
+a_{{1,0}}z^{2}+\beta^{2}a_{{4,0}}y^{2}\\
&\quad+(\beta^{4}a_{{6,0}}-\beta^{2}a_{{3,0}}+a_{{1,0}})x^{2}\Big),\\
g_1(x,y,z)&=\frac{1}{\beta^5}\Big(( a_{{3,0}}\mu_{{1}}\beta^{2}+a_{{3,0}}\alpha_{{1}}\beta^{2}
+2\mu_{{1}}a_{{4,0}}\beta^{2}+a_{{2,1}}\beta^{2}-2\mu_{{1}}a_{
{1,0}}+2\alpha_{{1}}a_{{1,0}})yz\\
&\quad+(a_{{5,0}}\mu_{{1}}
\beta^{3}+a_{{3,1}}\beta^{3}-2a_{{2,0}}\mu_{{1}}\beta-2a_{{1,1}
}\beta)xz+(2a_{{6,0}}\mu_{{1}}\beta^{4}+2a_{{6,0}
}\alpha_{{1}}\beta^{4}\\
&\quad+a_{{5,1}}\beta^{4}-2a_{{3,0}}\mu_{{1}}
\beta^{2}-2\mu_{{1}}a_{{4,0}}\beta^{2}-a_{{2,1}}\beta^{2}+2
\mu_{{1}}a_{{1,0}}-2\alpha_{{1}}a_{{1,0}})xy\\
&\quad+(a_{{2,0}}\mu_{{1}}\beta+a_{{1,1}}\beta)z^{2}
+(a_{{5,0}}\mu_{{1}}\beta^{3}+a_{{5,0}}\alpha_{{1}}\beta^{3}+a_{{4,1}}\beta^{3}
-a_{{2,0}}\mu_{{1}}\beta+\alpha_{{1}}a_{{2,0}}\beta)y^{2}\\
&\quad+(a_{{6,1}}\beta^{5}-a_{{5,0}}\mu_{{1}}\beta^{3}-a_{{3,1}}
\beta^{3}+a_{{2,0}}\mu_{{1}}\beta+a_{{1,1}}\beta)x^{2}\Big).\nonumber
\end{split}
\end{equation}
\end{small}

Note that system \eqref{eq-diff-1} is in the form of \eqref{equ-3}. Now computing the first-order averaged functions associated to system \eqref{eq-diff-1}, we obtain
\begin{small}
\begin{equation}\label{eq-diff-3}
\begin{split}
f_{1,1}(R,X_3)=-\frac{\pi R}{\beta^5}\bar{f}_{1,1}(R,X_3),\quad
f_{1,3}(R,X_3)=\frac{\pi}{\beta^5}\bar{f}_{1,3}(R,X_3),
\end{split}
\end{equation}
\end{small}
where
\begin{equation}\label{eq-diff-3-1}
\begin{split}
\bar{f}_{1,1}(R,X_3)&=(-\beta^{2}a_{{3,0}}+2a_{{1,0}})X_{{3}}-2\beta^{4}\alpha_{{1}},\\
\bar{f}_{1,3}(R,X_3)&=(\beta^{4}a_{{6,0}}-\beta^{2}a_{{3,0}}+\beta^{2}a_{{4,0}}
+a_{{1,0}})R^{2}+2a_{{1,0}}X_{{3}}^{2}+2\beta^{4}\mu_{{1}}X_{{3}}.\nonumber
\end{split}
\end{equation}
It is obvious that system \eqref{eq-diff-3} can have at most one real solution with $R>0$. Hence, system \eqref{eq-diff-1} can have at most one limit cycle bifurcating from the origin. Moreover, the determinant of the Jacobian of $(f_{1,1}(R,X_3),f_{1,3}(R,X_3))$ is
\begin{equation}\label{eq-diff-4}
\begin{split}
D_1(R,X_3)=\mbox{det}\left(
 \begin{matrix}
   \frac{\partial f_{1,1}}{\partial R} & \frac{\partial f_{1,1}}{\partial X_3}\\
   \frac{\partial f_{1,3}}{\partial R} & \frac{\partial f_{1,3}}{\partial X_3}
  \end{matrix}
  \right)
  =-\frac{2\pi^2}{\beta^{10}}\cdot\bar{D}_1(R,X_3),\nonumber
\end{split}
\end{equation}
where
\begin{equation}\label{eq-diff-4-1}
\begin{split}
\bar{D}_1(R,X_3)&=-2\beta^{8}\mu_{{1}}\alpha_{{1}}+(-2\beta^{2}a_{{1,0}}a
_{{3,0}}+4a_{{1,0}}^{2})X_{{3}}^{2}+(\beta^{6}a
_{{3,0}}a_{{6,0}}-2\beta^{4}a_{{1,0}}a_{{6,0}}\\
&\quad-\beta^{4}a_{{3,0}}^{2}+\beta^{4}a_{{3,0}}a_{{4,0}}+3\beta^{2}a_{{1,0}}a_{{3,0}}
-2\beta^{2}a_{{1,0}}a_{{4,0}}-2a_{{1,0}}^{2})R^{2}+(-\beta^{6}a_{{3,0}}\mu_{{1}}\\
&\quad+2\beta^{4}a_{{1,0}}\mu_{{1}}-4\beta^{4}a_{{1,0}}\alpha_{{1}})X_{{3}}.\nonumber
\end{split}
\end{equation}
It follows from Theorem \ref{semi-averaging} that system \eqref{eq-diff-1} can have one limit cycle bifurcating from the origin if the semi-algebraic system
\begin{equation}\label{eq-diff-5}
\begin{split}
\left\{
\begin{array}{ll}
&\bar{f}_{1,1}(R,X_3)=\bar{f}_{1,3}(R,X_3)=0, \\
&R>0,\quad \bar{D}_1(R,X_3)\neq0,\quad \beta\neq0
\end{array}
\right.
\end{split}
\end{equation}
has exactly one real solution with respect to
the variables $R, X_3$. Using {\sf{DISCOVERER}} (or the package RegularChains[SemiAlgebraicSetTools] in Maple), we find that system \eqref{eq-diff-5} has exactly one real solution if and only if one of the conditions $\mathcal{C}_i$ $(i=1,2,3,4)$ in \eqref{eqap-8} holds.

To let $(f_{1,1}(R,X_3),f_{1,3}(R,X_3))=(0,0)$, we take $\alpha_1=0$, $\mu_1=0$, $a_{3,0}=0$, $a_{1,0}=0$ and $a_{4,0}=-\beta^2a_{6,0}$. Computing the second-order averaged functions, we have
\begin{equation}\label{eq-diff-6}
\begin{split}
f_{2,1}(R,X_3)=\frac{\pi R}{4\beta^5}\bar{f}_{2,1}(R,X_3),\quad
f_{2,3}(R,X_3)=-\frac{\pi}{\beta^5}\bar{f}_{2,3}(R,X_3),
\end{split}
\end{equation}
where
\begin{equation}\label{eq-diff-6-1}
\begin{split}
\bar{f}_{2,1}&(R,X_3)=a_{{2,0}}a_{{6,0}}R^{2}+(4\beta^{2}a_{{3,1}}-8a_{{1,1}}
)X_{{3}}+8\beta^{4}\alpha_{{2}},\\
\bar{f}_{2,3}&(R,X_3)=-2a_{{1,1}}X_{{3}}^{2}+a_{{2,0}}a_{{6,0}}R^{2}X_{{3}}
-2\beta^{4}\mu_{{2}}X_{{3}}+(-\beta^{4}a_{{6,1}}+\beta^{2}a_{{3,1}
}-\beta^{2}a_{{4,1}}-a_{{1,1}})R^{2}.\nonumber
\end{split}
\end{equation}
It is not hard to check that the polynomial system $\{\bar{f}_{2,1}(R,X_3),\bar{f}_{2,3}(R,X_3)\}$ has at most two real solutions with $R>0$. As a result, system \eqref{eq-diff-1} can have at most two limit cycles bifurcating from the origin. In what follows, we show that this number can be reached.

Note that the determinant of the Jacobian of $(f_{2,1}(R,X_3),f_{2,3}(R,X_3))$ is
\begin{equation}\label{eq-diff-7}
\begin{split}
D_2(R,X_3)=\mbox{det}\left(
 \begin{matrix}
   \frac{\partial f_{2,1}}{\partial R} & \frac{\partial f_{2,1}}{\partial X_3}\\
   \frac{\partial f_{2,3}}{\partial R} & \frac{\partial f_{2,3}}{\partial X_3}
  \end{matrix}
  \right)
  =-\frac{\pi^2}{4\beta^{10}}\cdot\bar{D}_2(R,X_3),\nonumber
\end{split}
\end{equation}
where
\begin{equation}\label{eq-diff-7-1}
\begin{split}
\bar{D}_2(R,X_3)&=-16\beta^{8}\mu_{{2}}\alpha_{{2}}
+3R^{4}a_{{2,0}}^{2}a_{{6,0}}^{2}
+(-16\beta^{2}a_{{1,1}}a_{{3,1}}+32a_{{1,1}}^{2})X_{{3}}^{2}\\
&+(-4\beta^{2}a_{{2,0}}a_{{3,1}}a_{{6,0}}-4a_{{1,1}}a_{{2,0}}a_{{6,0}})R^{2}X_{{3}}
+(-8\beta^{6}a_{{3,1}}\mu_{{2}}+16\beta^{4}a_{{1,1}}\mu_{{2}}-32
\beta^{4}a_{{1,1}}\alpha_{{2}})X_{{3}}\\
&+(8\beta^{4}
a_{{2,0}}a_{{6,0}}\alpha_{{2}}-6\beta^{4}a_{{2,0}}a_{{6,0}}\mu_{{2
}}-16a_{{1,1}}^{2}-16\beta^{2}a_{{1,1}}a_{{4,1}}+24\beta^{
2}a_{{1,1}}a_{{3,1}}+8\beta^{4}a_{{3,1}}a_{{4,1}}\\
&-8\beta^{4}a_{{3,1}}^{2}-16\beta^{4}a_{{1,1}}a_{{6,1}}+8\beta^{6}a_{{3,1}
}a_{{6,1}})R^{2}.
\nonumber
\end{split}
\end{equation}

From Theorem \ref{semi-averaging}, we know that system \eqref{eq-diff-1} can have two limit cycles bifurcating from the origin if the semi-algebraic system
\begin{equation}\label{eq-diff-8}
\begin{split}
\left\{
\begin{array}{ll}
&\bar{f}_{2,1}(R,X_3)=\bar{f}_{2,3}(R,X_3)=0, \\
&R>0,\quad \bar{D}_2(R,X_3)\neq0,\quad \beta\neq0
\end{array}
\right.
\end{split}
\end{equation}
has exactly two real solutions with respect to
the variables $R, X_3$. In a similar way, we find that system \eqref{eq-diff-8} has exactly two real solutions if and only if one of the conditions $\bar{\mathcal{C}}_i$ $(i=1,2,3,4)$ in \eqref{eqap-10} holds. This completes the proof of Theorem \ref{main-th-3}.

\end{proof}

\section{Zero-Hopf Bifurcation in a Four-Dimensional Hyperchaotic \\ System}\label{sec-hyper}
The aim of this section is to study the existence of periodic orbits of the following four-dimensional differential system:
\begin{equation}\label{eqhyp-1}
\begin{split}
\dot{x}&=a_1(y-x)-a_2w,\quad \dot{y}=a_3x-y-b_1xz-a_4w,\\
\dot{z}&=b_2xy-a_5z,\quad \dot{w}=a_6yz-a_7w,
\end{split}
\end{equation}
where the $a_i$'s, $b_1$ and $b_2$ are real parameters. Note that system \eqref{eqhyp-1} is a generalization of the system introduced in \cite{DZYet19}, where the system is restricted to $b_1=b_2=1$. Here using the second-order averaging method, we study the maximum number of limit cycles that can bifurcate from a zero-Hopf equilibrium of system \eqref{eqhyp-1} when the system is perturbed inside the class of all differential systems of the same form. Thus $a_i$ and $b_j$ may take the form as follows
\begin{equation}\label{eqhyp-2}
\begin{split}
a_i=a_{i,0}+\varepsilon a_{i,1}+\varepsilon^2a_{i,2},\quad b_j=b_{j,0}+\varepsilon b_{j,1}+\varepsilon^2b_{j,2},
\end{split}
\end{equation}
where $i\in\{1,\ldots,7\}$ and $j\in\{1,2\}$, with $\varepsilon$ being a small parameter.

It is not hard to check that the origin is a zero-Hopf equilibrium of system \eqref{eqhyp-1} if we impose the following conditions
\begin{equation}\label{eqhyp-3}
\begin{split}
a_{1,0}=-1,\quad a_{3,0}=\beta^2+1,\quad a_{5,0}=a_{7,0}=0,
\end{split}
\end{equation}
where $\beta\neq0$. That is, under conditions \eqref{eqhyp-3}, the coefficient matrix of the linear part of system \eqref{eqhyp-1} at the origin has eigenvalues 0, 0 and $\pm\beta i$.

Now we investigate the periodic orbits bifurcating from the origin for system \eqref{eqhyp-1} under conditions \eqref{eqhyp-3}. By doing the linear change of variables $(x,y,z,w)\rightarrow(x_1,x_2,x_3,x_4)$:
\begin{equation}\label{eqhyp-4}
\begin{split}
x&=-\frac{1}{\beta^2}\Big((-a_{2,0}+a_{4,0})x_1+\beta a_{2,0} x_2+(a_{2,0}-a_{4,0})x_4\Big),\\
y&=-\frac{1}{\beta^2}\Big((-a_{2,0}(\beta^2+1)+a_{4,0})x_1+\beta a_{4,0} x_2+(a_{2,0}(\beta^2+1)-a_{4,0})x_4\Big),\\
z&=x_3-x_4,\\
w&=x_4,
\end{split}
\end{equation}
we write the linear part of system \eqref{eqhyp-1} at the origin when $\varepsilon=0$ into its real Jordan normal form. Then the differential system \eqref{eqhyp-1} becomes
\begin{equation}\label{eqhyp-5}
\begin{split}
\dot{x}_1&=-\beta x_2+g_1(\varepsilon,x_1,x_2,x_3,x_4),\\
\dot{x}_2&=\beta x_1+g_2(\varepsilon,x_1,x_2,x_3,x_4),\\
\dot{x}_3&=g_3(\varepsilon,x_1,x_2,x_3,x_4),\\
\dot{x}_4&=g_4(\varepsilon,x_1,x_2,x_3,x_4),\\
\end{split}
\end{equation}
where the expressions of $g_i$ for $i=1,\ldots,4$ are given in Appendix \ref{app-A}. Using the algorithmic tests presented in Section \ref{sub-sect3-3}, we obtain the following result on the number of limit cycles of system \eqref{eqhyp-1}.

\begin{theorem}\label{main-th-4}
The following statements hold for $\varepsilon>0$ sufficiently small.
\begin{itemize}
\item [\emph{(a)}] System \eqref{eqhyp-1} has, up to the first-order averaging, at most 1 limit cycle bifurcating from the origin, and this number can be reached if one of the following 16 conditions holds:
\begin{equation}\label{eqhyp-6}
\begin{split}
\mathcal{T}_1&=[a_{6,0},\,b_{2,0},\,a_{5,1},\,a_{1,1}<0;\,\,(\beta^2+1)a_{2,0}-a_{4,0}>0],\\
&\,\qquad\qquad\qquad\qquad\qquad\vdots\\
\mathcal{T}_9&=[a_{6,0},\,a_{5,1},\,a_{1,1},\,(\beta^2+1)a_{2,0}-a_{4,0}<0;\,\,b_{2,0}>0],\\
&\,\qquad\qquad\qquad\qquad\qquad\vdots\\
\mathcal{T}_{16}&=[a_{6,0},\,b_{2,0},\,a_{5,1},\,a_{1,1},\,(\beta^2+1)a_{2,0}-a_{4,0}>0].
\end{split}
\end{equation}

\item [\emph{(b)}] System \eqref{eqap-7} has, up to the second-order averaging, at most 2 limit cycles bifurcating from the origin, and this number can be reached if one of the following 32 conditions holds:
\begin{equation}\label{eqhyp-7}
\begin{split}
\bar{\mathcal{T}}_1&=[b_{1,0},\,a_{2,0},\,b_{2,1},\,a_{5,2},\,a_{1,2},\,\bar{R}_7<0;\,\, \bar{R}_6\geq0],\\
&\qquad\qquad\qquad\qquad\qquad\vdots\\
\bar{\mathcal{T}}_{17}&=[a_{2,0},\,b_{2,1},\,a_{5,2},\,a_{1,2}<0;\,\,\bar{R}_6\leq0;\,\,b_{1,0},\,\bar{R}_7>0],\\
&\qquad\qquad\qquad\qquad\qquad\vdots\\
\bar{\mathcal{T}}_{32}&=[\bar{R}_6\geq0;\,\,b_{1,0},\,a_{2,0},\,b_{2,1},\,a_{5,2},\,a_{1,2},\,\bar{R}_7>0],
\end{split}
\end{equation}
where
\begin{equation}\label{eqhyp-8}
\begin{split}
\bar{R}_6&=(\beta^2+1)a_{2,1}+a_{2,0}a_{3,1},\\
\bar{R}_7&=a_{{2,1}}^{2}a_{{6,0}}\beta^{4}+\big(2a_{{2,0}}a_{{2,1}}a_{{3
,1}}a_{{6,0}}-4a_{{1,2}}a_{{2,0}}b_{{1,0}}+2a_{{2,1}}^{2}a_{{6,0}}\big)\beta^{2}
+a_{{2,0}}^{2}a_{{3,1}}^{2}a_{{6,0}}\\
&\quad+2a_{{2,0}}a_{{2,1}}a_{{3,1}}a_{{6,0}}
+a_{{2,1}}^{2}a_{{6,0}}.\nonumber
\end{split}
\end{equation}

\end{itemize}
\end{theorem}
The complete conditions for \eqref{eqhyp-6} and \eqref{eqhyp-7} are shown in Appendix \ref{app-B}. Theorem \ref{main-th-4} can be proved by using similar calculations and arguments to the proof of Theorem \ref{main-th-3}. The details of its proof are omitted here.

\section{Proof of Theorem \ref{main-th-0}}\label{sec-B}
This section is devoted to the proof of Theorem \ref{main-th-0}. We first present a lemma which plays a key role in determining the numbers $\mu_{i,j}$ with $i\in\{1,2,\ldots,k\}$ and $j\in\{1,3,\ldots,n\}$.

\begin{lemma}\label{lem-B-1}
Let $k\geq1$ and $G(x_1,\ldots,x_n)=\frac{G_n(x_1,\ldots,x_n)}{x_1^k}$ be a rational function in $x_1,\ldots,x_n$, where $G_n(x_1,\ldots,x_n)$ is a polynomial in $\mathbb{R}[x_1,\ldots,x_n]$ whose total degree, denoted $\emph{deg}(G_n)$, is equal to $d$. Then
\begin{equation}\label{equB-1}
\begin{split}
\sum_{i_1,\ldots,i_L=1}^n\frac{\partial^LG(x_1,\ldots,x_n)}{\partial x_{i_1}\cdots\partial x_{i_L}}=\frac{\bar{G}_n(x_1,\ldots,x_n)}{x_1^{k+L}},
\end{split}
\end{equation}
where $\bar{G}_n(x_1,\ldots,x_n)$ is a polynomial in $\mathbb{R}[x_1,\ldots,x_n]$ of total degree no more than $d$, and $\partial^L$ denotes the $L$th Fr\'echet derivative of the function $G(x_1,\ldots,x_n)$.
\end{lemma}

\begin{proof}
Let $x_1^{j_1}x_2^{j_2}\cdots x_n^{j_n}$ be an arbitrary monomial of $G_n(x_1,\ldots,x_n)$ with all of its exponents $j_1,\ldots,j_n$ being nonnegative integers satisfying $j_1+\cdots+j_n\leq d$. Note that the operation
\begin{equation}\label{equB-2}
\begin{split}
\frac{\partial^L\left(x_1^{j_1}x_2^{j_2}\cdots x_n^{j_n}/x_1^k\right)}{\partial x_{i_1}\cdots\partial x_{i_L}}
\end{split}
\end{equation}
takes partial derivative with respect to $x_1$ once, the total degree of the resulting numerator remains the same, and the degree of the resulting denominator with respect to $x_1$ increases by one. On the other hand, the operation \eqref{equB-2} takes partial derivative with respect to $x_i$ ($i\neq1$) once, the resulting denominator remains the same, and the total degree of the resulting numerator decreases by one (or becomes 0). The desired result follows directly from the above fact.
\end{proof}

\begin{proof}[{\bf Proof of Theorem \ref{main-th-0}.}] In order to prove this theorem we need the following claim.
\smallskip

\noindent{{\bf Claim 1.}} The following relation holds:
\begin{equation}\label{equB-3}
\begin{split}
F_{k,j}=\frac{\bar{F}_{k,j}}{R^{k-1}},\quad j=1,3,\ldots,n,
\end{split}
\end{equation}
where each $\bar{F}_{k,j}$ is a polynomial in $R,X_3,\ldots,X_n$ of total degree at most $km$.

We only prove the claim for the case $j=1$ since similar arguments can be used in other cases. Recall that, given any real value $|\zeta|<1$, the following expansion holds
\begin{equation}\label{equB-4}
\begin{split}
\frac{1}{1+\zeta}=\sum_{h\geq0}(-1)^h\zeta^h.\nonumber
\end{split}
\end{equation}
Hence, the parametric formula \eqref{equ3-0-1} for $j=1$ can be written as
\begin{equation}\label{equB-5}
\begin{split}
\frac{dR}{d\theta}=\frac{1}{b}\left(\sum_{i=1}^kS_{i,1}\varepsilon^i\right)
\times\left[1+\sum_{h\geq1}\left(-\frac{1}{bR}\right)^h\left(S_{1,2}\varepsilon
+S_{2,2}\varepsilon^2+\cdots+S_{k,2}\varepsilon^k\right)^h\right],
\end{split}
\end{equation}
where the dependence on $(\theta,R,X_3,\ldots,X_n)$ is avoided to simplify the notation. Given $h^*\geq h$, we join together the coefficients of $\varepsilon^{h^*}$ in $\big(\sum_{i=1}^kS_{i,2}\varepsilon^i\big)^h$, which corresponds to summing up to $h^*$ with $h$ indices $i_1,i_2,\ldots,i_h\geq1$; that is,
\begin{equation}\label{equB-6}
\begin{split}
\Big(\sum_{i=1}^kS_{i,2}\varepsilon^i\Big)^h=\sum_{h^*\geq h}\varepsilon^{h^*}\left(\sum_{i_1+i_2+\cdots+i_h=h^*}S_{i_1,2}S_{i_2,2}\cdots S_{i_h,2}\right).
\end{split}
\end{equation}
Using \eqref{equB-6}, we have
\begin{equation}\label{equB-7}
\begin{split}
\sum_{h\geq1}\left(-\frac{1}{bR}\right)^h\Big(\sum_{i=1}^kS_{i,2}\varepsilon^i\Big)^h
&=\sum_{h\geq1}\sum_{h^*\geq h}\varepsilon^{h^*}\left(-\frac{1}{bR}\right)^h\sum_{i_1+i_2+\cdots+i_h=h^*}S_{i_1,2}S_{i_2,2}\cdots S_{i_h,2}\\
&=\sum_{h^*\geq 1}\sum_{h=1}^{h^*}\varepsilon^{h^*}\left(-\frac{1}{bR}\right)^h\sum_{i_1+i_2+\cdots+i_h=h^*}S_{i_1,2}S_{i_2,2}\cdots S_{i_h,2},\nonumber
\end{split}
\end{equation}
where the order of the summation in the indices has been changed. In order to simplify the notation further, we define the following auxiliary function
\begin{equation}\label{equB-8}
\begin{split}
\Delta_{h^*}(\theta,R,X_3,\ldots,X_n):=\sum_{h=1}^{h^*}\left(-\frac{1}{bR}\right)^h\sum_{i_1+i_2+\cdots+i_h=h^*}S_{i_1,2}S_{i_2,2}\cdots S_{i_h,2}.
\end{split}
\end{equation}
Substituting \eqref{equB-8} in the differential equation \eqref{equB-5}, we have
\begin{equation}\label{equB-9}
\begin{split}
\frac{dR}{d\theta}&=\frac{1}{b}\left(\sum_{i=1}^kS_{i,1}\varepsilon^i\right)
\times\left(1+\sum_{h^*\geq1}\varepsilon^{h^*}\Delta_{h^*}\right)\\
&=\frac{1}{b}\left[\sum_{i=1}^kS_{i,1}\varepsilon^i
+\sum_{k\geq2}\left(\sum_{h^*=1}^{k-1}S_{k-h^*,1}\Delta_{h^*}\right)\varepsilon^k\right]\\
&=\frac{1}{b}\left[S_{1,1}\varepsilon+\sum_{k\geq2}\left(S_{k,1}
+\sum_{h^*=1}^{k-1}S_{k-h^*,1}\Delta_{h^*}\right)\varepsilon^k\right].
\end{split}
\end{equation}
Hence, we have
\begin{equation}\label{equB-10}
\begin{split}
F_{1,1}:=\frac{S_{1,1}}{b},\quad F_{k,1}:=\frac{1}{b}\left(S_{k,1}
+\sum_{h^*=1}^{k-1}S_{k-h^*,1}\Delta_{h^*}\right),\quad k\geq2,
\end{split}
\end{equation}
where $\Delta_{h^*}$ is the function defined in \eqref{equB-8}.

Since $S_{i,1}$ and $S_{i,2}$ are polynomials of total degree at
most $m$ in $R,X_3,\ldots,X_n$ for $i=1,2,\ldots,k$, the function $\bar{\Delta}_{h^*}:=\Delta_{h^*}\cdot R^{h^*}$ is a polynomial of total degree at most $h^*m$ in $R,X_3,\ldots,X_n$. Using this fact, we complete the proof of the claim from equation \eqref{equB-10}.

Let $\boldsymbol{f}_i(\boldsymbol{\eta})=(f_{i,1}(\boldsymbol{\eta}),f_{i,3}(\boldsymbol{\eta}),\ldots,f_{i,n}(\boldsymbol{\eta}))$ be the $i$th-order averaged function associated to system \eqref{equ3-0-1}, with $i\in\{1,2,\ldots,k\}$. In the following we show that for each $j\in\{1,3,\ldots,n\}$, there exists a smallest nonnegative integer $\mu_{i,j}\leq i-1$ such that $R^{\mu_{i,j}}f_{i,j}(\boldsymbol{\eta})\in\mathcal{R}[\boldsymbol{\eta}]$. Moreover, we prove that $R^{\mu_{i,j}}f_{i,j}(\boldsymbol{\eta})$ is a polynomial in $\mathcal{R}[\boldsymbol{\eta}]$ of total degree no more than $im$.

Case $i=1$. Since $F_{1,j}$ is a polynomial of total degree at most $m$ in $R,X_3,\ldots,X_n$, it follows from \eqref{equ2-4} that $y_{1,j}(\theta,\boldsymbol{\eta})$ is a polynomial of total degree at most $m$ in $R,X_3,\ldots,X_n$.

By induction hypothesis, $y_{i,j}$ is a rational function of the form
\begin{equation}\label{equB-11}
\begin{split}
y_{i-1,j}(\theta,\boldsymbol{\eta})=\frac{\bar{y}_{i-1,j}(\theta,\boldsymbol{\eta})}{R^{i-2}},\quad i=2,\ldots,k,
\end{split}
\end{equation}
where $\bar{y}_{i-1,j}(\theta,\boldsymbol{\eta})$ is a polynomial of total degree at most $(i-1)m$ in $R,X_3,\ldots,X_n$.

To simplify notations, let $\boldsymbol{\eta}=(R,X_3,\ldots,X_n):=(\eta_1,\eta_2,\ldots,\eta_{n-1})$, $\boldsymbol{F}_{i}=(F_{i,1},F_{i,3},\ldots,F_{i,n}):=(F_{i,1},F_{i,2},\ldots,F_{i,n-1})$ and $\boldsymbol{y}_{i}=(y_{i,1},y_{i,3},\ldots,y_{i,n}):=(y_{i,1},y_{i,2},\ldots,y_{i,n-1})$. Now refer to the expression of $\boldsymbol{y}_{i}$ given in \eqref{equ2-4}, there only appear the previous functions $\boldsymbol{y}_{j_1}$, for $1\leq j_1\leq i-1$. Using equation \eqref{equ2-4} and the $L$-multilinear map defined in \eqref{equ2-2}, for any given integer $\ell$ with $1\leq\ell\leq i-1$ and each $j_2\in\{1,2,\ldots,n-1\}$, we have the following summation function
\begin{equation}\label{equB-12}
\begin{split}
y_{i,j_2}(\theta,\boldsymbol{\eta}) \quad & \dashrightarrow \quad \sum_{S_{\ell}}\partial^LF_{i-\ell,j_2}(\theta,\boldsymbol{\eta})
\bigodot_{j_1=1}^{\ell}\boldsymbol{y}_{j_1}(\theta,\boldsymbol{\eta})^{b_{j_1}}\\
&=\sum_{S_{\ell}}\sum_{i_1,\ldots,i_L=1}^{n-1}\frac{\partial^LF_{i-\ell,j_2}}
{\partial\eta_{i_1}\cdots\partial\eta_{i_L}}y_{1,i_1}\cdots y_{1,i_{b_1}}y_{2,i_{b_1+1}}\cdots y_{2,i_{b_2}}\cdots y_{\ell,i_{b_1+\cdots+b_{\ell-1}+1}}\cdots y_{\ell,i_{b_1+\cdots+b_{\ell}}}\\
&=\sum_{S_{\ell}}\sum_{i_1,\ldots,i_L=1}^{n-1}\frac{\partial^LF_{i-\ell,j_2}}
{\partial\eta_{i_1}\cdots\partial\eta_{i_L}}
\frac{\bar{y}_{1,i_1}\cdots\bar{y}_{1,i_{b_1}}\bar{y}_{2,i_{b_1+1}}\cdots\bar{y}_{2,i_{b_1+b_2}}\bar{y}_{\ell,i_{b_1
+\cdots+b_{\ell-1}+1}}\cdots\bar{y}_{\ell,i_{b_1+\cdots+b_{\ell}}}}{R^{b_2+2b_3+\cdots+(\ell-1)b_{\ell}}}.
\end{split}
\end{equation}
It follows from equation \eqref{equB-3} and Lemma \ref{lem-B-1} that the denominator of the expression \eqref{equB-12} is a polynomial in $R$ with degree at most
\[i-\ell-1+L+b_2+2b_3+\cdots+(\ell-1)b_{\ell}=i-1,\]
where the equalities $b_1+b_2+\cdots+b_{\ell}=L$ and $b_1+2b_2+\cdots+\ell b_{\ell}=\ell$ have been used. Moreover, from Lemma \ref{lem-B-1} and the induction hypothesis \eqref{equB-11}, we conclude that the numerator of the expression \eqref{equB-12} is a polynomial in $R,X_3,\ldots,X_n$ of total degree at most
\[(i-\ell)m+b_1m+2b_2m+\cdots+\ell b_{\ell}m=im.\]
This completes the proof of Theorem \ref{main-th-0}.
\end{proof}

\section{Proof of Theorem \ref{main-th-1}}\label{sect-main-1}
Applying Lemma \ref{lem-main-1} to system \eqref{equ-3} with $k=1$, we obtain the standard form
\begin{equation}\label{equ3-5}
\begin{split}
\frac{dR}{d\theta}&=\varepsilon F_1(\theta,R,X_3,\ldots,X_n)+\mathcal{O}(\varepsilon^2),\\
\frac{dX_s}{d\theta}&=\varepsilon F_s(\theta,R,X_3,\ldots,X_n)+\mathcal{O}(\varepsilon^2),\quad s=3,\ldots,n,
\end{split}
\end{equation}
where
\begin{equation}\label{equ3-6}
\begin{split}
F_1&=\frac{1}{b}\Big(a_1R+\sum_{i_1+\cdots+i_n=m}(p_{1,i_1,i_2,\ldots,i_n,0}\cos\theta
+p_{2,i_1,i_2,\ldots,i_n,0}\sin\theta)
(R\cos\theta)^{i_1}(R\sin\theta)^{i_2}X_3^{i_3}\cdots X_n^{i_n}\Big),\\
F_s&=\frac{1}{b}\Big(c_{s,1}X_s+\sum_{i_1+\cdots+i_n=m} p_{s,i_1,i_2,\ldots,i_n,0}(R\cos\theta)^{i_1}(R\sin\theta)^{i_2}X_3^{i_3}\cdots X_n^{i_n}\Big).\nonumber
\end{split}
\end{equation}
To apply the averaging method, we take in the standard form \eqref{equ2-1} $\boldsymbol{x}=\boldsymbol{\eta}=(R,X_3,\ldots,X_n)$, $t=\theta$, $T=2\pi$, and
\[\boldsymbol{F}(\theta,\boldsymbol{\eta})=(F_1(\theta,\boldsymbol{\eta}),F_3(\theta,\boldsymbol{\eta}),
\ldots,F_n(\theta,\boldsymbol{\eta})).\]
We need to compute the first-order averaged functions
\begin{equation}\label{equ3-7}
\begin{split}
f_{1,s}(\boldsymbol{\eta})=\int_0^{2\pi}F_s(\theta,\boldsymbol{\eta})d\theta,\quad s=1,3,\ldots,n;
\end{split}
\end{equation}
for this we define the integral equations
\begin{equation}\label{equ3-8}
\begin{split}
I_{i,j}=\int_0^{2\pi}\cos^i\theta\sin^j\theta d\theta,\quad i,j\in\mathbb{N}.
\end{split}
\end{equation}
In what follows, we divide our discussions into two cases.

{\bf Case 1:} $m\geq2$ is even. From equations \eqref{equ3-5}, \eqref{equ3-7} and \eqref{equ3-8}, we obtain after some analyses and calculations that
\begin{equation}\label{equ3-9}
\begin{split}
f_{1,1}&(\boldsymbol{\eta})=\frac{R}{ b}\Bigg[a_1I_{0,0}+\sum_{j=3}^n\Big(p_{1,1,m-2,\boldsymbol{e}_j,0}I_{2,m-2}
+p_{1,3,m-4,\boldsymbol{e}_j,0}I_{4,m-4}+\cdots+p_{1,m-1,0,\boldsymbol{e}_j,0}I_{m,0}\\
&\quad+p_{2,m-2,1,\boldsymbol{e}_j,0}I_{m-2,2}+p_{2,m-4,3,\boldsymbol{e}_j,0}I_{m-4,4}+\cdots
+p_{2,0,m-1,\boldsymbol{e}_j,0}I_{0,m}\Big)R^{m-2}X_j\\
&\quad+\sum_{3\leq j_1\leq j_2\leq j_3\leq n}\Big(p_{1,1,m-4,\boldsymbol{e}_{j_1j_2j_3},0}I_{2,m-4}
+p_{1,3,m-6,\boldsymbol{e}_{j_1j_2j_3},0}I_{4,m-6}+\cdots\\
&\quad+p_{1,m-3,0,\boldsymbol{e}_{j_1j_2j_3},0}I_{m-2,0}
+p_{2,m-4,1,\boldsymbol{e}_{j_1j_2j_3},0}I_{m-4,2}
+p_{2,m-6,3,\boldsymbol{e}_{j_1j_2j_3},0}I_{m-6,4}\\
&\quad+\cdots+p_{2,0,m-3,\boldsymbol{e}_{j_1j_2j_3},0}I_{0,m-2}\Big)R^{m-4}X_{j_1}X_{j_2}X_{j_3}+\cdots\\
&\quad+\sum_{3\leq j_1\leq j_2\leq\cdots\leq j_{m-1}\leq n}\Big(p_{1,1,0,\boldsymbol{e}_{j_1j_2\cdots j_{m-1}},0}I_{2,0}
+p_{2,0,1,\boldsymbol{e}_{j_1j_2\cdots j_{m-1}},0}I_{0,2}\Big)X_{j_1}X_{j_2}\cdots X_{j_{m-1}}\Bigg],\\
&=\frac{R}{ b}\Bigg[a_1I_{0,0}+\sum_{j=3}^n\sum_{i=0}^{\frac{m-2}{2}}\Big(p_{1,2i+1,m-2i-2,\boldsymbol{e}_j,0}I_{2i+2,m-2i-2}
+p_{2,2i,m-2i-1,\boldsymbol{e}_j,0}I_{2i,m-2i}\Big)R^{m-2}X_j\\
&\quad+\sum_{3\leq j_1\leq j_2\leq j_3\leq n}\sum_{i=0}^{\frac{m-4}{2}}\Big(p_{1,2i+1,m-2i-4,\boldsymbol{e}_{j_1j_2j_3},0}I_{2i+2,m-2i-4}
+p_{2,2i,m-2i-3,\boldsymbol{e}_{j_1j_2j_3},0}I_{2i,m-2i-2}\Big)\\
&\quad\times R^{m-4}X_{j_1}X_{j_2}X_{j_3}+\cdots+\sum_{3\leq j_1\leq j_2\leq\cdots\leq j_{m-1}\leq n}\Big(p_{1,1,0,\boldsymbol{e}_{j_1j_2\cdots j_{m-1}},0}I_{2,0}\\
&\quad+p_{2,0,1,\boldsymbol{e}_{j_1j_2\cdots j_{m-1}},0}I_{0,2}\Big)X_{j_1}X_{j_2}\cdots X_{j_{m-1}}\Bigg]\\
&=\frac{R}{b}\bar{f}_{1,1}(\boldsymbol{\eta})
\end{split}
\end{equation}
and
\begin{equation}\label{equ3-10}
\begin{split}
f_{1,s}(\boldsymbol{\eta})&=\frac{1}{b}\Bigg[c_{s,1}X_sI_{0,0}
+\Big(\sum_{i=0}^{\frac{m}{2}}p_{s,2i,m-2i,\boldsymbol{0}_{n-2},0}I_{2i,m-2i}\Big)R^m\\
&\quad+\sum_{3\leq j_1\leq j_2\leq n}\Big(\sum_{i=0}^{\frac{m-2}{2}}p_{s,2i,m-2i-2,\boldsymbol{e}_{j_1j_2},0}I_{2i,m-2i-2}\Big)R^{m-2}X_{j_1}X_{j_2}\\
&\quad+\cdots+\sum_{3\leq j_1\leq j_2\leq\cdots\leq j_m\leq n}p_{s,0,0,\boldsymbol{e}_{j_1j_2\cdots j_m},0}I_{0,0}X_{j_1}X_{j_2}\cdots X_{j_m}\Bigg]\\
&=\frac{1}{b}\bar{f}_{1,s}(\boldsymbol{\eta}),\quad s=3,\ldots,n,
\end{split}
\end{equation}
where $\boldsymbol{e}_{j}\in\mathbb{N}^{n-2}$ is the unit vector whose $j$th entry equal to 1, and $\boldsymbol{e}_{j_1j_2\cdots j_m}\in\mathbb{N}^{n-2}$ has the sum of the $j_1$th, the $j_2$th, $\ldots$, and the $j_m$th entries equal to $m$ and the others equal to 0 (these entries can coincide).

Now we apply the averaging theorem to obtain limit cycles of system \eqref{equ3-5}. Let $\mathcal{S}_1=\{\bar{f}_{1,1}(\boldsymbol{\eta}),\bar{f}_{1,3}(\boldsymbol{\eta}),\ldots,\bar{f}_{1,n}(\boldsymbol{\eta})\}$ be the algebraic system formed by the expressions in square brackets of \eqref{equ3-9} and \eqref{equ3-10}. By B\'ezout's theorem, the maximum number of common zeros that $\mathcal{S}_1$ can have is $(m-1)\cdot m^{n-2}$. In general this upper bound cannot be reached, see our Table \ref{Tab-1} for some concrete examples in dimension three.

{\bf Case 2:} $m\geq3$ is odd. Using arguments similar to those in the even case, we obtain
\begin{equation}\label{equ3-9-2}
\begin{split}
f_{1,1}&(\boldsymbol{\eta})=\frac{R}{ b}\Bigg[a_1I_{0,0}+\sum_{i=0}^{\frac{m-1}{2}}\Big(p_{1,2i+1,m-2i-1,\boldsymbol{0}_{n-2},0}I_{2i+2,m-2i-1}
+p_{2,2i,m-2i,\boldsymbol{0}_{n-2},0}I_{2i,m-2i+1}\Big)R^{m-1}\\
&\quad+\sum_{3\leq j_1\leq j_2\leq n}\sum_{i=0}^{\frac{m-3}{2}}\Big(p_{1,2i+1,m-2i-3,\boldsymbol{e}_{j_1j_2},0}I_{2i+2,m-2i-3}
+p_{2,2i,m-2i-2,\boldsymbol{e}_{j_1j_2},0}I_{2i,m-2i-1}\Big)\\
&\quad\times R^{m-3}X_{j_1}X_{j_2}+\cdots+\sum_{3\leq j_1\leq j_2\leq\cdots\leq j_{m-1}\leq n}\Big(p_{1,1,0,\boldsymbol{e}_{j_1j_2\cdots j_{m-1}},0}I_{2,0}\\
&\quad+p_{2,0,1,\boldsymbol{e}_{j_1j_2\cdots j_{m-1}},0}I_{0,2}\Big)X_{j_1}X_{j_2}\cdots X_{j_{m-1}}\Bigg]\\
&=\frac{R}{b}\bar{f}_{1,1}(\boldsymbol{\eta})
\end{split}
\end{equation}
and
\begin{equation}\label{equ3-10-2}
\begin{split}
f_{1,s}(\boldsymbol{\eta})&=\frac{1}{b}\Bigg[c_{s,1}X_sI_{0,0}
+\sum_{j=3}^n\Big(\sum_{i=0}^{\frac{m-1}{2}}p_{s,2i,m-2i-1,\boldsymbol{e}_{j},0}I_{2i,m-2i-1}\Big)R^{m-1}X_j\\
&\quad+\sum_{3\leq j_1\leq j_2\leq j_3\leq n}\Big(\sum_{i=0}^{\frac{m-3}{2}}p_{s,2i,m-2i-3,\boldsymbol{e}_{j_1j_2j_3},0}I_{2i,m-2i-3}\Big)R^{m-3}X_{j_1}X_{j_2}X_{j_3}\\
&\quad+\cdots+\sum_{3\leq j_1\leq j_2\leq\cdots\leq j_m\leq n}p_{s,0,0,\boldsymbol{e}_{j_1j_2\cdots j_m},0}I_{0,0}X_{j_1}X_{j_2}\cdots X_{j_m}\Bigg]\\
&=\frac{1}{b}\bar{f}_{1,s}(\boldsymbol{\eta}),\quad s=3,\ldots,n.
\end{split}
\end{equation}

Similarly, we let $\mathcal{S}_2=\{\bar{f}_{1,1}(\boldsymbol{\eta}),\bar{f}_{1,3}(\boldsymbol{\eta}),\ldots,\bar{f}_{1,n}(\boldsymbol{\eta})\}$ be the algebraic system formed by the expressions in square brackets of \eqref{equ3-9-2} and \eqref{equ3-10-2}. By B\'ezout's theorem, the maximum number of common zeros that $\mathcal{S}_2$ can have is $(m-1)\cdot m^{n-2}$ as well. In short, we conclude that, for $\varepsilon>0$ sufficiently small, any differential system of the form \eqref{equ-3} up to the first-order averaging can have at most $(m-1)\cdot m^{n-2}$ limit cycles in a neighborhood of the origin. This completes the proof of Theorem \ref{main-th-1}.




\section*{Acknowledgments}
The authors wish to thank the referees for their insightful and helpful comments on an early version of the paper. This work is partially supported by the National Natural Science Foundation of China (NSFC 12101032 and NSFC 12131004).

\appendix

\section{Expressions for \eqref{eqhyp-5}}\label{app-A}
The expressions of $g_i$ in \eqref{eqhyp-5} for $i=1,\ldots,4$ are as follows.
\begin{equation}\label{eqhyp-5-1}
\begin{split}
g_1&=\left(\varepsilon A_1+\varepsilon^2A_2\right)x_1
+\left(\varepsilon A_3+\varepsilon^2A_4\right)x_2+\left(\varepsilon A_5+\varepsilon^2A_6\right)x_4+(A_{7,0}+\varepsilon A_{7,1})\\
&\quad\times(x_1x_3+x_4^2-x_1x_4-x_3x_4)
+(A_{8,0}+\varepsilon A_{8,1})(x_2x_3-x_2x_4)+\mathcal{O}(\varepsilon^2),\\
g_2&=\left(\varepsilon B_1+\varepsilon^2 B_2\right)x_1+\left(\varepsilon B_3+\varepsilon^2 B_4\right)x_2
+\left(\varepsilon B_5+\varepsilon^2 B_6\right)x_4+(B_{7,0}+\varepsilon B_{7,1})\\
&\quad\times(x_1x_3+x_4^2-x_1x_4-x_3x_4)
+(B_{8,0}+\varepsilon B_{8,1})(x_2x_3-x_2x_4)+\mathcal{O}(\varepsilon^2),\\
g_3&=-(\varepsilon a_{5,1}+\varepsilon^2a_{5,2})x_3
+\left(\varepsilon(a_{5,1}-a_{7,1})+\varepsilon^2(a_{5,2}-a_{7,2})\right)x_4
+(C_{1,0}+\varepsilon C_{1,1})x_1^2\\
&\quad+(C_{2,0}+\varepsilon C_{2,1})x_1x_2+(C_{3,0}+\varepsilon C_{3,1})x_1x_3+(C_{4,0}+\varepsilon C_{4,1})x_2^2\\
&\quad+(C_{5,0}+\varepsilon C_{5,1})x_2x_3+(C_{6,0}+\varepsilon C_{6,1})x_2x_4+(C_{7,0}+\varepsilon C_{7,1})x_3x_4\\
&\quad+(C_{8,0}+\varepsilon C_{8,1})x_1x_4+(C_{9,0}+\varepsilon C_{9,1})x_4^2+\mathcal{O}(\varepsilon^2),\\
g_4&=-(\varepsilon a_{7,1}+\varepsilon^2 a_{7,2})x_4+(D_{1,0}+\varepsilon D_{1,1})(x_1x_3+x_4^2-x_1x_4-x_3x_4)\\
&\quad+(D_{2,0}+\varepsilon D_{2,1})(x_2x_3-x_2x_4)+\mathcal{O}(\varepsilon^2),\nonumber
\end{split}
\end{equation}
where
\begin{align}\label{eqhyp-5-2}
A_1&=-\frac{a_{{2,0}}}{\xi}\big(\beta^{2}a_{{4,0}}a_{{1,1}}-a_{{2,0}}a_{{3
,1}}+a_{{3,1}}a_{{4,0}}\big),\quad
A_2=-\frac{a_{{2,0}}}{\xi}\big( \beta^{2}a_{{4,0}}a_{{1,2}}-a_{{2,0}}a_{{3,2}}+a_{{3,2}}a_{{4,0}}\big),\nonumber\\
A_3&=-\frac{\beta}{\xi}\big(a_{{1,1}}a_{{2,0}}a_{{4,0}}-a_{{1,1}}a_{{4,0}}^{2}
+a_{{2,0}}^{2}a_{{3,1}}\big),\quad
A_4=-\frac{\beta}{\xi}\big(a_{{1,2}}a_{{2,0}}a_{{4,0}}-a_{{1,2}}a_{{4,0}}^{2}
+a_{{2,0}}^{2}a_{{3,2}}\big),\nonumber\\
A_5&=\frac{1}{\xi}\Big(\big( a_{{1,1}}a_{{2,0}}a_{{4,0}}-a_{{2,0}}^{2}a_{{7,1}}-a_{{2,0}}
a_{{4,1}}+a_{{2,1}}a_{{4,0}}\big)\beta^{2}-a_{{2,0}}^{2}a_{{3,1}}-a_{{2,0}}^{2}a_{{7,1}}\nonumber\\
&\quad+a_{{2,0}}a_{{3,1}}a_{{4,0}}+2a_{{2,0}}a_{{4,0}}a_{{7,1}}
-a_{{4,0}}^{2}a_{{7,1}}\Big),\nonumber\\
A_6&=\frac{1}{\xi}
\Big(\big(a_{{1,2}}a_{{2,0}}a_{{4,0}}-a_{{2,0}}^{2}a_{{7,2}}-a_{{2,0}}
a_{{4,2}}+a_{{2,2}}a_{{4,0}}\big)\beta^{2}-a_{{2,0}}^{2}a_{{3,2}}-a_{{2,0}}^{2}a_{{7,2}}\nonumber\\
&\quad+a_{{2,0}}a_{{3,2}}a_{{4,0}}+2a_{{2,0}}a
_{{4,0}}a_{{7,2}}-a_{{4,0}}^{2}a_{{7,2}}\Big),\nonumber\\
A_{7,0}&=\frac{1}
{\beta^{2}\xi}
\Big(a_{{2,0}}^{3}a_{{6,0}}\beta^{4}+\big(2a_{{2,0}}^{3}a_{{6,0}
}-3a_{{2,0}}^{2}a_{{4,0}}a_{{6,0}}+a_{{2,0}}a_{{4,0}}^{2}a_{{6,0
}}-a_{{2,0}}^{2}b_{{1,0}}+a_{{2,0}}a_{{4,0}}b_{{1,0}}\big)\beta^{2}\nonumber\\
&\quad+a_{{2,0}}^{3}a_{{6,0}}-3a_{{2,0}}^{2}a_{{4,0}}a_{{6,0}}+3
a_{{2,0}}a_{{4,0}}^{2}a_{{6,0}}-a_{{4,0}}^{3}a_{{6,0}}\Big),\nonumber\\
A_{8,0}&=-\frac{1}{\beta\xi}
\Big(\big(a_{{2,0}}^{2}a_{{4,0}}a_{{6,0}}-a_{{2,0}}^{2}b_{{1,0}}
\big)\beta^{2}+a_{{2,0}}^{2}a_{{4,0}}a_{{6,0}}-2a_{{2,0}}a_
{{4,0}}^{2}a_{{6,0}}+a_{{4,0}}^{3}a_{{6,0}}\Big),\nonumber\\
A_{7,1}&=\frac{1}{\beta^{2}\xi}
\Big(a_{{2,0}}^{3}a_{{6,1}}\beta^{4}+\big(2a_{{2,0}}^{3}a_{{6,1}
}-3a_{{2,0}}^{2}a_{{4,0}}a_{{6,1}}+a_{{2,0}}a_{{4,0}}^{2}a_{{6,1
}}-a_{{2,0}}^{2}b_{{1,1}}+a_{{2,0}}a_{{4,0}}b_{{1,1}}\big)\beta^{2}\nonumber\\
&\quad+a_{{2,0}}^{3}a_{{6,1}}-3a_{{2,0}}^{2}a_{{4,0}}a_{{6,1}}+3
a_{{2,0}}a_{{4,0}}^{2}a_{{6,1}}-a_{{4,0}}^{3}a_{{6,1}}\Big),\nonumber\\
A_{8,1}&=-\frac{1}{\beta\xi}\Big(\big( a_{{2,0}}^{2}a_{{4,0}}a_{{6,1}}-a_{{2,0}}^{2}b_{{1,1}}
\big)\beta^{2}+a_{{2,0}}^{2}a_{{4,0}}a_{{6,1}}-2a_{{2,0}}a_
{{4,0}}^{2}a_{{6,1}}+a_{{4,0}}^{3}a_{{6,1}}\Big),\nonumber\\
B_1&=-\frac{1}{\beta\xi}\Big(a_{{1,1}}a_{{2,0}}^{2}\beta^{4}+\big( a_{{1,1}}a_{{2,0}}^{2}-a
_{{1,1}}a_{{2,0}}a_{{4,0}}\big)\beta^{2}-a_{{2,0}}^{2}a_{{3,1}
}+2a_{{2,0}}a_{{3,1}}a_{{4,0}}-a_{{3,1}}a_{{4,0}}^{2}\Big),\nonumber\\
B_2&=-\frac{1}{\beta\xi}\Big(a_{{1,2}}a_{{2,0}}^{2}\beta^{4}+\big( a_{{1,2}}a_{{2,0}}^{2}-a_{{1,2}}a_{{2,0}}a_{{4,0}}\big)\beta^{2}
-a_{{2,0}}^{2}a_{{3,2}}+2a_{{2,0}}a_{{3,2}}a_{{4,0}}-a_{{3,2}}a_{{4,0}}^{2}\Big),\nonumber\\
B_3&=-\frac{a_{{2,0}}-a_{{4,0}}}{\xi}\big(\beta^{2}a_{{1,
1}}a_{{2,0}}+a_{{1,1}}a_{{2,0}}-a_{{1,1}}a_{{4,0}}+a_{{2,0}}a_{{3,1}}\big),\nonumber\\
B_4&=-\frac{a_{{2,0}}-a_{{4,0}}}{\xi}\big(\beta^{2}a_{{1,
2}}a_{{2,0}}+a_{{1,2}}a_{{2,0}}-a_{{1,2}}a_{{4,0}}+a_{{2,0}}a_{{3,2}}\big),\nonumber\\
B_5&=\frac{1}
{\beta\xi}
\Big(\big(a_{{1,1}}a_{{2,0}}^{2}+a_{{2,0}}a_{{2,1}}\big)\beta^{4}
+\big(a_{{1,1}}a_{{2,0}}^{2}-a_{{1,1}}a_{{2,0}}a_{{4,0}}+a_{{2,0
}}a_{{2,1}}-a_{{2,0}}a_{{4,1}}-a_{{2,1}}a_{{4,0}}\nonumber\\
&\quad+a_{{4,0}}a_{{4,1}}\big)\beta^{2}-a_{{2,0}}^{2}a_{{3,1}}+2a_{{2,0}}a_{{3,1}}a_{
{4,0}}-a_{{3,1}}a_{{4,0}}^{2}\Big),\nonumber\\
B_6&=\frac{1}
{\beta\xi}
\Big(\big(a_{{1,2}}a_{{2,0}}^{2}+a_{{2,0}}a_{{2,2}}\big)\beta^{4
}+\big(a_{{1,2}}a_{{2,0}}^{2}-a_{{1,2}}a_{{2,0}}a_{{4,0}}+a_{{2,0
}}a_{{2,2}}-a_{{2,0}}a_{{4,2}}-a_{{2,2}}a_{{4,0}}\nonumber\\
&\quad+a_{{4,0}}a_{{4,2}}\big)\beta^{2}-a_{{2,0}}^{2}a_{{3,2}}
+2a_{{2,0}}a_{{3,2}}a_{{4,0}}-a_{{3,2}}a_{{4,0}}^{2}\Big),\nonumber\\
B_{7,0}&=-\frac{b_{{1,0}}\left(a_{{2,0}}-a_{{4,0}}\right)^{2}}
{\beta\xi},\quad
B_{8,0}=\frac{a_{{2,0}}b_{{1,0}}\left(a_{{2,0}}-a_{{4,0}}\right)}
{\xi},\quad
B_{7,1}=\frac{b_{1,1}}{b_{{1,0}}}B_{7,0},\quad
B_{8,1}=\frac{b_{1,1}}{b_{{1,0}}}B_{8,0},\nonumber\\
C_{1,0}&=\frac{b_{{2,0}}\left(a_{{2,0}}-a_{{4,0}}\right)\lambda}{\beta^{4}},\quad
C_{2,0}=-\frac{b_{{2,0}}\left((\beta^{2}+1)a_{{2,0}}^{2}-a_{{4,0}}^{2}\right) }{\beta^{3}},\quad
C_{3,0}=\frac{a_{{6,0}}\lambda}{\beta^{2}},\nonumber\\
C_{4,0}&=\frac{a_{2,0}a_{4,0}b_{2,0}}{\beta^2},\quad
C_{5,0}=-\frac{a_{4,0}a_{6,0}}{\beta},\quad
C_{6,0}=\frac{1}{\beta^{3}}\Big(\big(a_{{2,0}}^{2}b_{{2,0}}+a_{{4,0}}a_{{6,0
}}\big)\beta^{2}+a_{{2,0}}^{2}b_{{2,0}}-a_{{4,0}}^{2}b_{{2,0}}\Big),\nonumber\\
C_{7,0}&=-C_{3,0},\quad
C_{8,0}=-\frac{\lambda
\left(\beta^{2}a_{{6,0}}+2a_{{2,0}}b_{{2,0}}-2a_{{4,0}}b_{{2,0
}}\right)}{\beta^{4}},\quad
C_{9,0}=\frac{\lambda
\left(\beta^{2}a_{{6,0}}+a_{{2,0}}b_{{2,0}}-a_{{4,0}}b_{{2,0}}\right) }{\beta^{4}},\nonumber\\
C_{1,1}&=\frac{b_{{2,1}}\left(a_{{2,0}}-a_{{4,0}}\right)
\lambda}{\beta^{4}},\quad
C_{2,1}=\frac{b_{2,1}}{b_{2,0}}C_{2,0},\quad
C_{3,1}=\frac{a_{{6,1}}\lambda}{\beta^{2}},\quad
C_{4,1}=\frac{a_{2,0}a_{4,0}b_{2,1}}{\beta^2},\nonumber\\
C_{5,1}&=-\frac{a_{4,0}a_{6,1}}{\beta},\quad
C_{6,1}=\frac{\big(a_{{2,0}}^{2}b_{{2,1}}+a_{{4,0}}a_{{6,1}}\big)\beta^
{2}+a_{{2,0}}^{2}b_{{2,1}}-a_{{4,0}}^{2}b_{{2,1}}}{\beta^{3}},\quad
C_{7,1}=-C_{3,1},\nonumber\\
C_{8,1}&=-\frac{\lambda
\left(\beta^{2}a_{{6,1}}+2a_{{2,0}}b_{{2,1}}-2a_{{4,0}}b_{{2,1}}\right)}
{\beta^{4}},\quad
C_{9,1}=\frac{\lambda
\left(\beta^{2}a_{{6,1}}+a_{{2,0}}b_{{2,1}}-a_{{4,0}}b_{{2,1}}
\right)}{\beta^{4}},\nonumber\\
D_{1,0}&=\frac{a_{{6,0}}\lambda}{\beta^{2}},\quad
D_{1,1}=\frac{a_{{6,1}}\lambda}{\beta^{2}},\quad
D_{2,0}=-\frac{a_{4,0}a_{6,0}}{\beta},\quad D_{2,1}=-\frac{a_{4,0}a_{6,1}}{\beta},\nonumber
\end{align}
with
\[\xi=\beta^{2}a_{{2,0}}^{2}+(a_{{2,0}}-a_{4,0})^{2},\quad \lambda=(\beta^2+1)a_{2,0}-a_{4,0}.\]

\section{Conditions for \eqref{eqhyp-6} and \eqref{eqhyp-7}}\label{app-B}

\begin{align}\label{eqhyp-6-1-0}
\mathcal{T}_1&=[a_{6,0},\,b_{2,0},\,a_{5,1},\,a_{1,1}<0;\,\,(\beta^2+1)a_{2,0}-a_{4,0}>0],\nonumber\\
\mathcal{T}_2&=[a_{6,0},\,b_{2,0},\,a_{5,1},\,(\beta^2+1)a_{2,0}-a_{4,0}<0;\,\,a_{1,1}>0],\nonumber\\
\mathcal{T}_3&=[a_{6,0},\,b_{2,0},\,a_{1,1},\,(\beta^2+1)a_{2,0}-a_{4,0}<0;\,\,a_{5,1}>0],\nonumber\\
\mathcal{T}_4&=[a_{6,0},\,b_{2,0}<0;\,\,a_{5,1},\,a_{1,1},\,(\beta^2+1)a_{2,0}-a_{4,0}>0],\nonumber\\
\mathcal{T}_5&=[b_{2,0},\,a_{5,1},\,a_{1,1},\,(\beta^2+1)a_{2,0}-a_{4,0}<0;\,\,a_{6,0}>0],\nonumber\\
\mathcal{T}_6&=[b_{2,0},\,a_{5,1}<0;\,\,a_{6,0},\,a_{1,1},\,(\beta^2+1)a_{2,0}-a_{4,0}>0],\nonumber\\
\mathcal{T}_7&=[b_{2,0},\,a_{1,1}<0;\,\,a_{6,0},\,a_{5,1},\,(\beta^2+1)a_{2,0}-a_{4,0}>0],\nonumber\\
\mathcal{T}_8&=[b_{2,0},\,(\beta^2+1)a_{2,0}-a_{4,0}<0;\,\,a_{6,0},\,a_{5,1},\,a_{1,1}>0],\nonumber\\
\mathcal{T}_9&=[a_{6,0},\,a_{5,1},\,a_{1,1},\,(\beta^2+1)a_{2,0}-a_{4,0}<0;\,\,b_{2,0}>0],\nonumber\\
\mathcal{T}_{10}&=[a_{6,0},\,a_{5,1}<0;\,\,b_{2,0},\,a_{1,1},\,(\beta^2+1)a_{2,0}-a_{4,0}>0],\nonumber\\
\mathcal{T}_{11}&=[a_{6,0},\,a_{1,1}<0;\,\,b_{2,0},\,a_{5,1},\,(\beta^2+1)a_{2,0}-a_{4,0}>0],\nonumber\\
\mathcal{T}_{12}&=[a_{6,0},\,(\beta^2+1)a_{2,0}-a_{4,0}<0;\,\,b_{2,0},\,a_{5,1},\,a_{1,1}>0],\nonumber\\
\mathcal{T}_{13}&=[a_{5,1},\,a_{1,1}<0;\,\,a_{6,0},\,b_{2,0},\,(\beta^2+1)a_{2,0}-a_{4,0}>0],\nonumber\\
\mathcal{T}_{14}&=[a_{5,1},\,(\beta^2+1)a_{2,0}-a_{4,0}<0;\,\,a_{6,0},\,b_{2,0},\,a_{1,1}>0],\nonumber\\
\mathcal{T}_{15}&=[a_{1,1},\,(\beta^2+1)a_{2,0}-a_{4,0}<0;\,\,a_{6,0},\,b_{2,0},\,a_{5,1}>0],\nonumber\\
\mathcal{T}_{16}&=[a_{6,0},\,b_{2,0},\,a_{5,1},\,a_{1,1},\,(\beta^2+1)a_{2,0}-a_{4,0}>0].\nonumber\\
\bar{\mathcal{T}}_1&=
[b_{1,0},\,a_{2,0},\,b_{2,1},\,a_{5,2},\,a_{1,2},\,\bar{R}_7<0;\,\, \bar{R}_6\geq0],\nonumber\\
\bar{\mathcal{T}}_2&=
[b_{1,0},\,a_{2,0},\,b_{2,1},\,a_{5,2}<0;\,\,
\bar{R}_6\geq0;\,\,a_{1,2},\,\bar{R}_7>0],\nonumber\\
\bar{\mathcal{T}}_3&=[b_{1,0},\,a_{2,0},\,b_{2,1},\,a_{1,2},\,\bar{R}_7<0;\,\,
\bar{R}_6\leq 0;\,\,a_{5,2}>0],\nonumber\\
\bar{\mathcal{T}}_4&=
[b_{1,0},\,a_{2,0},\,b_{2,1}<0;\,\,\bar{R}_6\leq 0;\,\,a_{5,2},\,a_{1,2},\,\bar{R}_7>0],\nonumber\\
\bar{\mathcal{T}}_5&=[b_{1,0},\,a_{2,0},\,a_{5,2},\,a_{1,2},\,\bar{R}_7<0;\,\,
\bar{R}_6\leq 0;\,\,b_{2,1}>0],\nonumber\\
\bar{\mathcal{T}}_6&=
[b_{1,0},\,a_{2,0},\,a_{5,2}<0;\,\,\bar{R}_6\leq 0;\,\,b_{2,1},\,a_{1,2},\,\bar{R}_7>0],\nonumber\\
\bar{\mathcal{T}}_7&=
[b_{1,0},\,a_{2,0},\,a_{1,2},\,\bar{R}_7<0;\,\, \bar{R}_6\geq0;\,\,b_{2,1},\,a_{5,2}>0],\nonumber\\
\bar{\mathcal{T}}_8&=
[b_{1,0},\,a_{2,0}<0;\,\, \bar{R}_6\geq0;\,\,b_{2,1},\,a_{5,2},\,a_{1,2},\,\bar{R}_7>0],\nonumber\\
\bar{\mathcal{T}}_9&=
[b_{1,0},\,b_{2,1},\,a_{5,2},\,a_{1,2}<0;\,\,\bar{R}_6\leq0;\,\,a_{2,0},\,\bar{R}_7>0],\nonumber\\
\bar{\mathcal{T}}_{10}&=
[b_{1,0},\,b_{2,1},\,a_{5,2},\,\bar{R}_7<0;\,\,\bar{R}_6\leq0;\,\,a_{2,0},\,a_{1,2}>0],\nonumber\\
\bar{\mathcal{T}}_{11}&
=[b_{1,0},\,b_{2,1},\,a_{1,2}<0;\,\,\bar{R}_6\geq0;\,\,a_{2,0},\,a_{5,2},\,\bar{R}_7>0],\nonumber\\
\bar{\mathcal{T}}_{12}&
=[b_{1,0},\,b_{2,1},\,\bar{R}_7<0;\,\,\bar{R}_6\geq0;\,\,a_{2,0},\,a_{5,2},\,a_{1,2}>0],\nonumber\\
\bar{\mathcal{T}}_{13}&
=[b_{1,0},\,a_{5,2},\,a_{1,2}<0;\,\,0\leq\bar{R}_6;\,\,a_{2,0},\,b_{2,1},\,\bar{R}_7>0],\nonumber\\
\bar{\mathcal{T}}_{14}&
=[b_{1,0},\,a_{5,2},\,\bar{R}_7<0;\,\,0\leq\bar{R}_6;\,\,a_{2,0},\,b_{2,1},\,a_{1,2}>0],\nonumber\\
\bar{\mathcal{T}}_{15}&
=[b_{1,0},\,a_{1,2}<0;\,\,\bar{R}_6\leq0;\,\,a_{2,0},\,b_{2,1},\,a_{5,2},\,\bar{R}_7>0],\nonumber\\
\bar{\mathcal{T}}_{16}&
=[b_{1,0},\,\bar{R}_7<0;\,\,\bar{R}_6\leq0;\,\,a_{2,0},\,b_{2,1},\,a_{5,2},\,a_{1,2}>0],\nonumber\\
\bar{\mathcal{T}}_{17}&
=[a_{2,0},\,b_{2,1},\,a_{5,2},\,a_{1,2}<0;\,\,\bar{R}_6\leq0;\,\,b_{1,0},\,\bar{R}_7>0],\nonumber\\
\bar{\mathcal{T}}_{18}&
=[a_{2,0},\,b_{2,1},\,a_{5,2},\,\bar{R}_7<0;\,\,\bar{R}_6\leq0;\,\,b_{1,0},\,a_{1,2}>0],\nonumber\\
\bar{\mathcal{T}}_{19}&
=[a_{2,0},\,b_{2,1},\,a_{1,2}<0;\,\,\bar{R}_6\geq0;\,\,b_{1,0},\,a_{5,2},\,\bar{R}_7>0],\nonumber\\
\bar{\mathcal{T}}_{20}&
=[a_{2,0},\,b_{2,1},\,\bar{R}_7<0;\,\,\bar{R}_6\geq0;\,\,b_{1,0},\,a_{5,2},\,a_{1,2}>0],\nonumber\\
\bar{\mathcal{T}}_{21}&
=[a_{2,0},\,a_{5,2},\,a_{1,2}<0;\,\,\bar{R}_6\geq0;\,\,b_{1,0},\,b_{2,1},\,\bar{R}_7>0],\nonumber\\
\bar{\mathcal{T}}_{22}&
=[a_{2,0},\,a_{5,2},\,\bar{R}_7<0;\,\,\bar{R}_6\geq0;\,\,b_{1,0},\,b_{2,1},\,a_{1,2}>0],\nonumber\\
\bar{\mathcal{T}}_{23}&
=[a_{2,0},\,a_{1,2}<0;\,\,\bar{R}_6\leq0;\,\,b_{1,0},\,b_{2,1},\,a_{5,2},\,\bar{R}_7>0],\nonumber\\
\bar{\mathcal{T}}_{24}&
=[a_{2,0},\,\bar{R}_7<0;\,\,\bar{R}_6\leq0;\,\,b_{1,0},\,b_{2,1},\,a_{5,2},\,a_{1,2}>0],\nonumber\\
\bar{\mathcal{T}}_{25}&
=[b_{2,1},\,a_{5,2},\,a_{1,2},\,\bar{R}_7<0;\,\,\bar{R}_6\geq0;\,\,b_{1,0},\,a_{2,0}>0],\nonumber\\
\bar{\mathcal{T}}_{26}&
=[b_{2,1},\,a_{5,2}<0;\,\,\bar{R}_6\geq0;\,\,b_{1,0},\,a_{2,0},\,a_{1,2},\,\bar{R}_7>0],\nonumber\\
\bar{\mathcal{T}}_{27}&
=[b_{2,1},\,a_{1,2},\,\bar{R}_7<0;\,\,\bar{R}_6\leq 0;\,\,b_{1,0},\,a_{2,0},\,a_{5,2}>0],\nonumber\\
\bar{\mathcal{T}}_{28}&
=[b_{2,1}<0;\,\,\bar{R}_6\leq 0;\,\,b_{1,0},\,a_{2,0},\,a_{5,2},\,a_{1,2},\,\bar{R}_7>0],\nonumber\\
\bar{\mathcal{T}}_{29}&
=[a_{5,2},\,a_{1,2},\,\bar{R}_7<0;\,\,\bar{R}_6\leq 0;\,\,b_{1,0},\,a_{2,0},\,b_{2,1}>0],\nonumber\\
\bar{\mathcal{T}}_{30}&
=[a_{5,2}<0;\,\,\bar{R}_6\leq 0;\,\,b_{1,0},\,a_{2,0},\,b_{2,1},\,a_{1,2},\,\bar{R}_7>0],\nonumber\\
\bar{\mathcal{T}}_{31}&
=[a_{1,2},\,\bar{R}_7<0;\,\, \bar{R}_6\geq0;\,\,b_{1,0},\,a_{2,0},\,b_{2,1},\,a_{5,2}>0],\nonumber\\
\bar{\mathcal{T}}_{32}&
=[\bar{R}_6\geq0;\,\,b_{1,0},\,a_{2,0},\,b_{2,1},\,a_{5,2},\,a_{1,2},\,\bar{R}_7>0].\nonumber
\end{align}

\end{document}